\documentclass[11pt]{article}

\usepackage{epsfig,epsf,fancybox}
\usepackage{amsmath}
\usepackage{mathrsfs}
\usepackage{amssymb}
\usepackage{graphicx}
\usepackage{color}
\usepackage[linktocpage,pagebackref,colorlinks,linkcolor=blue,anchorcolor=blue,citecolor=blue,urlcolor=blue]{hyperref}
\usepackage{boxedminipage}
\usepackage{stmaryrd}
\usepackage{multirow}
\usepackage{booktabs}

\newtheorem{theorem}{Theorem}[section]
\newtheorem{corollary}[theorem]{Corollary}
\newtheorem{lemma}[theorem]{Lemma}
\newtheorem{proposition}[theorem]{Proposition}
\newtheorem{definition}{Definition}[section]

\newtheorem{assumption}[theorem]{Assumption}

\newcommand{\ex}{{\bf\sf E}}
\newcommand{\st}{\textnormal{s.t.}}

\newcommand{\sign}{\textnormal{sign}}

\,
\,

\newcommand{\RR}{\mathbf R}

\newcommand{\be}{\begin{equation}}
\newcommand{\ee}{\end{equation}}
\newcommand{\ba}{\begin{array}}
\newcommand{\ea}{\end{array}}
\newcommand{\bpm}{\begin{pmatrix}}
\newcommand{\epm}{\end{pmatrix}}

\newcommand{\bea}{\begin{eqnarray}}
\newcommand{\eea}{\end{eqnarray}}
\newcommand{\beaa}{\begin{eqnarray*}}
\newcommand{\eeaa}{\end{eqnarray*}}

\begin{document}

\title{Iteration Bounds for Finding the $\epsilon$-Stationary Points \\ for Structured Nonconvex Optimization}

\author{
Bo Jiang
\thanks{Research Center for Management Science and Data Analytics, School of Information Management and Engineering, Shanghai University of Finance and Economics, Shanghai 200433, China. Email: isyebojiang@gmail.com. Research of this author was supported in part by National Science Foundation of China (Grant 11401364).} \and
Shuzhong Zhang
\thanks{Department of Industrial and Systems Engineering, University of Minnesota, Minneapolis, MN 55455, USA. Email: zhangs@umn.edu. Research of this author was supported in part by National Science Foundation (Grant CMMI-1161242).}}

\date{\today}

\maketitle

\begin{abstract}
In this paper we study proximal conditional-gradient (CG) and proximal gradient-projection type algorithms for a block-structured constrained nonconvex optimization model, which arises naturally from tensor data analysis. First, we introduce a new notion of $\epsilon$-stationarity, which is suitable for the structured problem under consideration. 
We then propose two types of first-order algorithms for the model based on the proximal conditional-gradient (CG) method and the proximal gradient-projection method respectively. If the nonconvex objective function is in the form of mathematical expectation, we then discuss how to incorporate randomized sampling to avoid computing the expectations exactly. For the general block optimization model, the proximal subroutines are performed for each block according to either the block-coordinate-descent (BCD) or the maximum-block-improvement (MBI) updating rule. If the gradient of the nonconvex part of the objective $f$ satisfies $\| \nabla f(x) - \nabla f(y)\|_q \le M \|x-y\|_p^\delta$ where $\delta=p/q$ with $1/p+1/q=1$, then we prove that the new algorithms have an overall iteration complexity bound of $O(1/\epsilon^q)$ in finding an $\epsilon$-stationary solution. If $f$ is concave then the iteration complexity reduces to $O(1/\epsilon)$. Our numerical experiments for tensor approximation problems show promising performances of the new solution algorithms.

\vspace{1cm}


\vspace{0.50cm}


\end{abstract}

\newpage

\section{Introduction}
The first-order algorithms and their iteration complexity analysis for nonconvex optimization problems have recently attracted considerable research attention; see e.g.~\cite{Nesterov04,CartisGouldToint10,CartisGouldToint11,CartisGouldToint13a,CartisGouldToint13b,GhadimiLanZhang13,GhadimiLan13}. In this paper,
we aim to solve the following block-structured nonconvex and nonsmooth optimization problem:
\begin{equation} \label{noncvx-block-opt}
\begin{array}{lll} \min & f(x_1,\cdots, x_d) + \sum_{i=1}^{d}h_i(x_i) \\ \st & x_i \in S_i \subseteq\RR^{n_i}, \quad i=1,\ldots,d, &  \end{array}
\end{equation}
where $f$ is differentiable but possibly nonconvex, and $h_i$ is convex but possibly nonsmooth, for $i=1,\ldots, d$.
Such optimization models arise from a variety of applications. As an example, in statistics it is often desirable to find a {\em regulated}\/ least square solution where the least square term is nonconvex, which falls into this category. Consider, for instance, the so-called sparse tensor PCA problem, where one wishes to find the best sparse rank-one approximation for a given tensor. Mathematically, this can be formulated by the following optimization model~\cite{Allen12}:
\[
\begin{array}{ll}
\min & -{\cal A} (x_1, x_2, \cdots , x_d)+ \rho \sum_{i=1}^d \| x_i \|_1 \\
\st & x_i \in S_i = \{x \,|\, \|x\|_2^2 \le 1\},\, i=1,2,...,d,
\end{array}
\]
where ${\cal A}$ is a given $d$-dimensional tensor data. The $L_1$-norm in the objective is used to promote the sparsity of the vectors $x_i$, $i=1,2,...,d$. The above problem is clearly an instance of \eqref{noncvx-block-opt}.

Towards eventually solving \eqref{noncvx-block-opt}, we first consider the special case when there is only one block of variables:
\begin{eqnarray}\label{prob:nonconvex}
\begin{array}{lll} \min & \Phi(x):=f(x) + h(x) \\ \st & x \in S \subseteq\RR^n.&  \end{array}\end{eqnarray}
Again, here $f$ is assumed to be differentiable but possibly nonconvex, and $h$ is convex but possibly nonsmooth; $S$ is assumed to be a compact and convex set, and
$\Phi^*$ is the optimal value of~\eqref{prob:nonconvex}. We denote
\begin{equation}\label{radius}
\mbox{\rm diam}_p(S) = \max_{x,y\, \in S}\|x -y \|_p,
\end{equation}
where $\|x \|_p = (\sum_{i=1}^{n}|x_i|^p)^{1/p}$.
In this paper, we shall propose
two types of first-order algorithms to find a certain {\it $\epsilon$-solution}\/ for \eqref{prob:nonconvex}. A first question arises in this context: {\em Is there a reasonable definition of $\epsilon$-solution (or rather, $\epsilon$-stationary solution) to aim for?} For the smooth unconstrained version of~\eqref{prob:nonconvex}, i.e. $\Phi(x) = f(x)$ and $S = \RR^n$, the natural definition of $\epsilon$-stationary point is:
$$
\| \nabla \Phi(x) \|_2 = \| \nabla f(x) \|_2 \le \epsilon.
$$
Nesterov~\cite{Nesterov04} and Cartis et al.~\cite{CartisGouldToint11} showed that the gradient decent type method with some properly chosen step size needs $O({1}/{\epsilon^2})$ iterations to find such a point. Moreover, Cartis et al.~\cite{CartisGouldToint10} constructed an example to show that the $O({1}/{\epsilon^2})$ complexity is actually tight for the steepest-descent algorithm.

However, the case for the constrained nonconvex optimization is more complicated. In fact, there are multiple quality measurements for approximative stationary points.
Cartis et al.~\cite{CartisGouldToint13a} proposed the following measure:
$$
\chi_S(x) := \left| \min\limits_{x+d \in S, \|d\|_2\le 1}\nabla f(x)^{\top}d \right|.
$$
Furthermore, they showed that it requires no more than $O({1}/{\epsilon^2})$ iterations for their adaptive cubic regularization algorithm to find a point $x$
such that
\begin{equation}\label{measure-cartis}
\chi_S(x) \le \epsilon.
\end{equation}
Along a related but different line,
Ghadimi et al.~\cite{GhadimiLanZhang13} used the squared norm of the residual for some generalized projection to evaluate the quality of solution. Specifically, the residual at point $x$ is defined as
$$
P_{S}(x,\gamma):=\frac{1}{\gamma} (x - x^+),
$$
where
$$x^+ = \arg\min_{y \in S}{\nabla f(x)^{\top}y + \frac{1}{\gamma}V(y,x) + h(y)},$$
and $V$ is some {\it prox-function}. (We refer the interested reader to~\cite{GhadimiLanZhang13} for the details).
The authors proposed a projected gradient algorithm and proved that it will take no more than $O({1}/{\epsilon^2})$ iterations to achieve
\begin{equation}\label{measure-lan}
\| P_{S}(x,\gamma) \|_2^2 \le \epsilon.
\end{equation}

In this paper we consider the following new notion of
stationarity for the nonconvex and nondifferentiable optimization model~\eqref{prob:nonconvex}:
\begin{definition}
We call $x$ to be a stationary point of~\eqref{prob:nonconvex} if the following condition holds:
\begin{equation}\label{formu:statn-cond}
\nabla f(x)^{\top}(y-x) + h(y) - h(x) \ge 0 \quad \forall y\in S.
\end{equation}
\end{definition}
In fact, if $x$ is a local minimizer of~\eqref{prob:nonconvex}, then it must satisfy~\eqref{formu:statn-cond}.
To see this, we shall use a contradiction argument. Suppose that there exists some $y \in S$ such that
$\nabla f(x)^{\top}(y-x) + h(y) - h(x) < 0$. Denote $d=y-x$. Then the directional derivative along direction $d$ at point $x$ satisfies
\begin{eqnarray*}
(f+h)^\prime(x;d)&=&\lim_{\alpha \downarrow 0}\frac{f(x+\alpha d)+h(x+\alpha d)-f(x)-h(x)}{\alpha}\\
& \le & \lim_{\alpha \downarrow 0}\frac{f(x+\alpha d)-f(x)}{\alpha}+ \lim_{\alpha \downarrow 0}\frac{(1-\alpha)h(x)+\alpha h(x+ d)-h(x)}{\alpha}\\
& = & \nabla f(x)^{\top}(y-x) + h(y) - h(x) < 0,
\end{eqnarray*}
where the inequality is due to the convexity of $h$.
Consequently, $x$ cannot be a local optimal solution of problem~\eqref{prob:nonconvex}.
Thus based on condition~\eqref{formu:statn-cond}, we consider the following definition of approximative stationary solution.

\begin{definition}
We call $x$ to be an $\epsilon$-stationary point of~\eqref{prob:nonconvex} if
\begin{equation}\label{formu:epsilon-KKT}
\psi_S(x):=\nabla f(x)^{\top}(y-x) + h(y) - h(x) \ge - \epsilon\quad \forall y\in S.
\end{equation}
\end{definition}
A similar condition for $L_2$-$L_p$ optimization problem was considered in~\cite{GeHeHe14}; the relationship between conditions~\eqref{formu:epsilon-KKT},~\eqref{measure-cartis} and~\eqref{measure-lan}
will be discussed in Section~\ref{Sec:Pre}.
To proceed, let us make the following technical assumption on the smooth part of the objective $f(x)$ throughout this paper.
\begin{assumption}\label{assump-p} There exists some $p > 1$ and $\lambda >0$ such that
\begin{equation}\label{ineq:p-power}
f(y) \le f(x) + \nabla f(x)^{\top}(y-x) + \frac{\lambda}{2}\|y - x\|^p_p, \, \, \, \forall x,y \in S.
\end{equation}
\end{assumption}

Some comments about Assumption~\ref{assump-p} are in order here.
First, notice that if $f(x)$ is concave, then \eqref{ineq:p-power} holds true for any $p>0$ and $\lambda>0$.
Second, if the gradient of $f$ satisfies
\begin{equation}\label{p-q-continuous}
\|\nabla f(x) - \nabla f(y)\|_q^q \le  M \|x - y\|_p^p\quad\forall\;x,y \in S
\end{equation}
for some $p,q>1$ and $\frac{1}{p} + \frac{1}{q} = 1$,
then the function itself also satisfies \eqref{ineq:p-power}. To see this, we let $z=y-x$ and $g(\alpha) = f(x+ \alpha z)$. It follows that
\begin{eqnarray*}
f(y)-f(x) &=& \int_{0}^{1}g^{\prime}(\alpha) z\, d\alpha = \int_{0}^{1}\nabla f(x+\alpha z)^{\top}z\, d\alpha\\
&\le&\int_{0}^{1}\nabla f(x)^{\top}z\,d\alpha + \left|\int_{0}^{1}( \nabla f(x+\alpha z)-\nabla f(x) )^{\top}z \,d\alpha \right| \\
&\le& \int_{0}^{1}\nabla f(x)^{\top}z\, d\alpha + \int_{0}^{1}\| \nabla f(x+\alpha z)-\nabla f(x)\|_q \|z\|_p\, d\alpha\\
&\le& \nabla f(x)^{\top}z + M^{1/q} \|z\|_p^{1+\frac{p}{q}} \int_{0}^{1} \alpha^{\frac{p}{q}} \, d\alpha\\
&=& \nabla f(x)^{\top}z + \frac{M^{1/q}}{p} \|z\|_p^p,
\end{eqnarray*}
where the last equality follows from $\frac{1}{p} + \frac{1}{q} = 1$. Thus, the function with Lipschitz continuous gradient
automatically satisfies inequality~\eqref{ineq:p-power} for $p=q=2$.
In fact, condition~\eqref{p-q-continuous} reflects the degree of the H\"olderian continuity of $\nabla f$, which was also considered in~\cite{Devo-Fran-Nesterov-2013} to construct an inexact first order oracle.
Finally, we remark that
the $p$-th powered $p$-norm function:
\begin{equation}\label{multivar-poly}
f(x)=\sum_{i=1}^{n} x_i^p,\;\forall\; 1< p \le 2,
\end{equation}
on $\RR^n_+$ also satisfies \eqref{p-q-continuous}.
We observe that
the function is separable with respect to all $x_i$, and so it suffices to show that there exists some $\lambda$ such that:
\begin{equation}\label{univar-poly}
v^p \le u^{p} + (u^p)^\prime(v - u) + \frac{\lambda}{2} |v-u|^p  = u^{p} + p\, u^{p-1}(v - u) + \frac{\lambda}{2} |v-u|^p,
\end{equation}
when $1 < p \le 2$. If $u=0$, then the inequality trivially holds for any $\lambda \ge 2$; otherwise we can divide both sides by $|u|^p$ and aim to prove an equivalent formulation:
\begin{equation}\label{one-parameter}
k^p \le 1 + p(k-1) + \frac{\lambda}{2} |k-1|^p,
\end{equation}
where $k=v/u$. To this end, define
$$
g(k):=\left\{\begin{array}{cl}0, & \mbox{if}\;k = 1\\
\frac{k^p - 1 - p(k-1)}{|k-1|^p}, & \mbox{otherwise.}\end{array}\right.
$$
Observe that $\lim\limits_{k \to +\infty}g(k) 
=1$, and from L'Hospital's rule
$$
\lim\limits_{k \to 1}g(k) =\left\{\begin{array}{cl}0, & \mbox{if}\;1< p < 2 \\
1, & \mbox{if}\;p = 2, \end{array}\right.
$$
and so $g(k)$ is upper bounded on $\RR$ and there exits some $\hat{\lambda}$ such that~\eqref{one-parameter} holds. Finally by letting $\lambda = \max\{2,\hat{\lambda}\}$, the inequality~\eqref{univar-poly} follows.

In this paper we shall propose two algorithms for solving problem~\eqref{prob:nonconvex}, both achieving an $O({1}/{\epsilon^q})$ iteration complexity, where $\frac{1}{p} + \frac{1}{q} = 1$ and $p$ is the parameter in~\eqref{ineq:p-power}. As a result, a larger value of $p$ leads to a smaller value of $q$, hence a better iteration bound for the algorithm.
In other words, this result shows that the ``smoothness'' of the function will be reflected in the speed of the convergence. In particular, when $p=2$, we get $O({1}/{\epsilon^q}) = O({1}/{\epsilon^2})$, which is consistent with the result of Cartis et al.~\cite{CartisGouldToint13a}. Another extreme case is when $f(x)$ is concave, and in this case the complexity can be reduced to $O({1}/{\epsilon})$; we shall elaborate more on this point later.

The algorithms to be proposed in this paper use only the first-order information of $f$. When $p=2$ and $h$ does not appear, {\bf Algorithm 2} in this paper is simply the gradient projection algorithm. When $p=2$ and $f$ is {\em convex}, {\bf Algorithm 2} coincides with the so-called ISTA (iterative shrinkage-thresholding algorithm (\cite{BT09}).
Similarly, if $h$ does not appear, 
then {\bf Algorithm 1} coincides with the conditional gradient (CG) method, where the subproblem to be solved in each iteration involves a linear objective function. Note that the conditional gradient method was originally proposed by Frank and Wolfe~\cite{FrankWolfe56}, and recently has regained some research attention primarily due to the fact that the linear subproblem is easier to solve in the context of large scale optimization. To the best of our knowledge, the iteration complexity analysis for the CG method had only been established for convex optimization~\cite{FreundGrigas13,Lan13}. In other words, the current paper presents for the first time an iteration complexity bound for the CG method in the context of nonconvex optimization. A recent work related to the current paper is~\cite{LMDZ14}, which proposes a smoothing SQP method to solve \eqref{prob:nonconvex}. The key differences are: (1) In \cite{LMDZ14}, $f(x)$ is assumed to be in the form of $\|(Ax-b)_+\|_q^q$ where $0<q<1$, and the gradient of $h(x)$ is assumed to be Lipschitz continuous and $h$ may also be non-convex; (2) In \cite{LMDZ14} the constraint set $S$ is assumed to be polyhedral; (3) In \cite{LMDZ14} a convex quadratic program is solved at each step as a subroutine. As we can see, the basic assumptions on the problem setting as well as the subroutines applied are all very different. The results are also fundamentally different. In \cite{LMDZ14}, a different notion of $\epsilon$-KKT condition is introduced, and a smoothing SQP method is shown to reach an $\epsilon$-KKT point in no more than $O(\epsilon^{q-4})$ iterations. In the current paper, the $\epsilon$-stationarity condition is based on a variational inequality. Though $h(x)$ is assumed to be convex in our context, the CG subroutine may be much easier to solve.
In this paper, we also extend our studies to stochastic nonconvex optimization, for which a combination of randomized sampling method and the first-order approximation approach is proposed. We also consider the case where the nonsmooth part is concave and show that the CG method achieves similar complexity bound by properly incorporating some randomized smoothing scheme. Besides, we show that our approach can be modified to handle multi-block nonconvex optimization which covers a great variety of applications.

This paper is organized as follows. In Section~\ref{Sec:Pre} we introduce a new notion of $\epsilon$-stationary solution and discuss its relationships to those proposed by Cartis et al.\ and Ghadimi et al.\/.  
In Section~\ref{Sec:de-alg}, we present two algorithms for problem~\eqref{prob:nonconvex} and analyze their iteration complexity bounds. We then develop a stochastic algorithm and a randomized smoothing algorithm in Section~\ref{Sec:sto-alg}.
In Section~\ref{Sec:mul-alg}, the solution methods are extended to solve a nonconvex multi-block optimization model.
Finally, we present our numerical experiments in Section~\ref{numerical}. 

\section{The $\epsilon$-Stationarity Condition} \label{Sec:Pre}

\subsection{Relationship with Existing Results} 
In the last section, we introduced two quality measures for an approximate solution, denoted by $\chi_S(x)$ and $\|P_{S}(x,\gamma) \|_2^2$ respectively. We also  proposed our new quality measure $\psi_S(x)$ (formula~\eqref{formu:epsilon-KKT}). Obviously, these three measures are different, but they are related. To be precise, their relationship is summarized in the following proposition.

\begin{proposition}\label{prop:relation} 
(i) If $h(x)\equiv 0$ and $\psi_S(x) \ge -\epsilon$, then $\chi_S(x) \le \epsilon$;\\
(ii) Suppose the prox-function $V(y,x)=\|y-x\|_2^2/2$, then $\psi_S(x) \ge -\epsilon$ implies  $\|P_{S}(x,\gamma) \|_2^2 \le \frac{\epsilon}{\gamma} $. Conversely if we further assume the gradient function $\nabla f(x)$ is continuous,
then $\|P_{S}(x,\gamma) \|_2^2 \le \epsilon$ implies
$$\psi_S(x) \ge - (\gamma\, \tau + \gamma\, \varsigma + \mbox{\rm diam}_2(S))\sqrt{\epsilon},$$
where $\tau = \max_{x \in S} \|\nabla f(x)\|_2$, $\varsigma = \max_{x \in S}\min_{z \in \partial h(x)}\| z \|_2$ and $\mbox{\rm diam}_2(S)$ is defined in~\eqref{radius}.
\end{proposition}
\begin{proof} According to the definition,
\begin{eqnarray*}
&& \chi_S(x) \le \epsilon\\
& \Longleftrightarrow& \min\limits_{x+d \in S, \|d\|_2\le 1}\nabla f(x)^{\top}d \ge -\epsilon\\
& \Longleftrightarrow& \nabla f(x)^{\top}(y-x) \ge -\epsilon,\; \forall\; \|y-x\|_2 \le 1,\; y \in S.
\end{eqnarray*}
Therefore, (i) is readily implied by~\eqref{formu:epsilon-KKT} when $h(x) \equiv 0$. \\
Now, let us prove (ii). Since $V(y,x)=\|y-x\|_2^2/2$, one has
\begin{equation}\label{formu:general-projection}
\left(\nabla f(x)+\frac{1}{\gamma}(x^+ - x) +z \right)^{\top}(y - x^+) \ge 0\quad\forall\; y \in S,
\end{equation}
where $z \in \partial h(x^+)$. We can particularly choose $y = x$ and get
$$
\nabla f(x)^{\top}(x - x^+) + h(x) - h(x^+) \ge \left(\nabla f(x) +z \right)^{\top}(x - x^+) \ge \frac{1}{\gamma}\|x^+ - x\|_2^2.
$$
So if $\psi_S(x) \ge -\epsilon$ (i.e.~\eqref{formu:epsilon-KKT} holds) then we have $\|P_{S}(x,\gamma) \|_2^2 \le \frac{\epsilon}{\gamma} $. To show the other direction, note that $\varsigma$ is finite ($S$ is compact) and $h(x)$ is convex. Thus, for $x, x^+ \in S$ we can choose $w \in \partial h(x)$ such that
$$\varsigma\,  \|x - x^+\|_2 \ge w^{\top}(x - x^+) \ge h(x) - h(x^+).$$
This inequality together with~\eqref{formu:general-projection} implies that
\begin{eqnarray*}
&& \nabla f(x)^{\top}(y - x) + h(y) - h(x) + (\|\nabla f(x)\|_2  + \varsigma ) \|x - x^+\|_2\\
&\ge& \nabla f(x)^{\top}(y - x) + h(y) - h(x) + \nabla f(x)^{\top}(x - x^+) + h(x) - h(x^+) \\
&=& \nabla f(x)^{\top}(y - x^+) + h(y) - h(x^+) \\
&\ge& \left(\nabla f(x)+ z \right)^{\top}(y - x^+)\\
&\ge& - \frac{1}{\gamma}(x^+ - x)^{\top}(y - x^+)\\
&\ge& -\frac{1}{\gamma}\,\|y - x^+\|_2 \|x^+ - x\|_2
 \quad\forall\; y \in S,
\end{eqnarray*}
where $z \in \partial h(x^{+})$ and the second inequality follows from the convexity of $h(x)$. Furthermore, since $S$ is compact and $\nabla f(x)$ is continuous, by rearranging the terms in the above inequality, the final conclusion follows.
\end{proof}

Under the conditions of Proposition~\ref{prop:relation}, the relationship among these three measures for an approximate local optimal solution is depicted in Figure~\ref{graph}. According to this result, we see that condition~\eqref{formu:epsilon-KKT} is in some sense more general than~\eqref{measure-cartis} and~\eqref{measure-lan}. In the remainder of this paper, we shall refer the $\epsilon$-stationary condition to~\eqref{formu:epsilon-KKT}.

\begin{figure}[!htb]
\centering
  \includegraphics[scale=0.7]{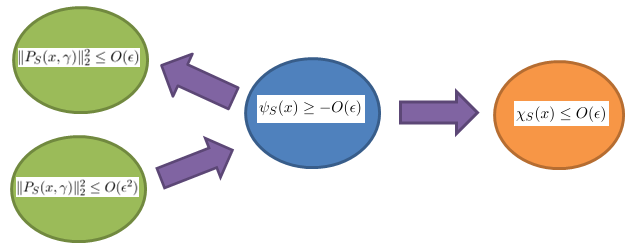}
  \caption{Relationship of the three different $\epsilon$-stationarity measures.}\label{graph}
\end{figure}

\subsection{Sufficient Conditions for $\epsilon$-Stationarity}
For a given point $z$, we define the following two functions, which will play a crucial role in our solution methods to be proposed later:
\begin{eqnarray}
L(x;z) &:=& f(z) + \nabla f(z)^{\top}(x-z) + h(x), \label{func:linear} \\
U(x;z) &:=& f(z) + \nabla f(z)^{\top}(x-z) + \frac{\lambda}{2}\|x-z \|^p_p + h(x). \label{func:upper-bound}
\end{eqnarray}
In fact,~\eqref{func:linear} is obtained by linearizing the smooth part of $\Phi$ and~\eqref{func:upper-bound} is an upper bound of $\Phi$ if Assumption~\ref{assump-p} holds.
These two functions lead to the following two convex optimization subroutines:
\begin{equation} \label{MLP}
\left\{ \begin{array}{ll} \min\limits_{x} & L(x;z) \\ \st & x \in S,
\end{array} \right.
\end{equation}
and
\begin{equation} \label{MUP}
\left\{ \begin{array}{ll} \min\limits_{x} & U(x;z) \\ \st & x \in S.
\end{array} \right.
\end{equation}

Denote ${z}_{L}$ and ${z}_{U}$ to be the minimizer of \eqref{MLP} 
and \eqref{MUP} 
respectively. The partial linearized improvement at point $z$ is defined by
$$
\triangle L_z :=  L(z;z) - L({z}_{L};z)=- \nabla f(z)^{\top}({z}_{L}-z) + h(z) - h(z_L).
$$
Similarly we define the partial $p$-powered improvement by
$$
\triangle U_{z} :=  U(z;z) - U({z}_{U};z)= - \nabla f(z)^{\top}({z}_{U}-z) - \frac{\lambda}{2}\|{z}_{U}-z \|^p_p+ h(z) - h(z_U).
$$
The following lemma, which is inspired by~\cite{GeHeHe14}, states that if the progress gained by solving \eqref{MLP} 
or
\eqref{MUP} 
is small, then we are already near a stationary point.
\begin{lemma}\label{kkt-criteria}Given $\epsilon \ge 0$, for any $z \in S$,\\
(i) if $\triangle L_z \le \epsilon$, then $z$ is an $\epsilon$-stationary point of~\eqref{prob:nonconvex};\\
(ii) if $\triangle U_{z} \le \frac{1}{2}\left(\frac{\epsilon}{\mbox{\rm diam}_p(S)\lambda^{1/p}}\right)^q$ with $\frac{1}{p} + \frac{1}{q}=1$, and $\epsilon \le \mbox{\rm diam}^p_p(S)\lambda$, then $z$ is an $\epsilon$-stationary point of~\eqref{prob:nonconvex}.
\end{lemma}
\begin{proof} Let us first consider (i). Since ${z}_{L}$ is optimal to \eqref{MLP},  we have 
$$
L(y;z) - L({z}_{L};z) = \nabla f(z)^{\top}(y-{z}_{L}) + h(y) - h(z_L) \ge 0,\;\forall\;y \in S.
$$
It follows that
\begin{eqnarray*}\label{ineq:lemma1}
&& \nabla f(z)^{\top}(y-z) + h(y) - h(z)\\
&=&\nabla f(z)^{\top}(y-{z}_{L}) +  h(y) - h(z_L) + \nabla f(z)^{\top}({z}_{L}-z) + h(z_L) - h(z) \\
&\ge& \nabla f(z)^{\top}({z}_{L}-z)+ h(z_L) - h(z)\quad \forall\; y \in S.
\end{eqnarray*}
Then by definition, $\triangle L_z \le \epsilon$ implies that
$$\nabla f(z)^{\top}(y-z)+ h(y) - h(z) \ge - \triangle L_z \ge - \epsilon.$$
To prove statement (ii), we consider the point $y = z + s({z}_{L} - z)$ with $0 \le s \le 1$. Clearly, by definition of partial $p$-powered improvement and convexity of $h$, it follows that
\begin{eqnarray*}
\frac{1}{2}\left(\frac{\epsilon}{\mbox{\rm diam}_p(S)\lambda^{1/p}}\right)^q
& \ge  & \triangle U_{z}  \ge   U(z;z) - U(y; z)\\
&=& \nabla f(z)^{\top}(z - y) - \frac{\lambda}{2}\|y-z\|^p_p + h(z) - h(y)\\
&\ge &-s \nabla f(z)^{\top}({z}_{L}-z) - \frac{\lambda}{2}s^p\|{z}_{L}-z\|^p_p + s \left(h(z)-h({z}_{L})\right).
\end{eqnarray*}
Letting $s = \left(\frac{\epsilon}{\mbox{\rm diam}^p_p(S)\lambda}\right)^{\frac{1}{p-1}}$ and rearranging the terms in the above inequality yield
\begin{eqnarray*}
 \triangle L_z  &=& -\nabla f(z)^{\top}({z}_{L}-z) + h(z) - h(z_L) \\
&\le& \frac{\lambda}{2}s^{p-1} \|{z}_{L}-z\|^p_p + \frac{1}{2s}\left(\frac{\epsilon}{\mbox{\rm diam}_p(S)\lambda^{1/p}}\right)^q \\
&=& \frac{\lambda}{2}\frac{\epsilon}{\mbox{\rm diam}^p_p(S)\lambda} \|{z}_{L}-z\|^p_p + \frac{1}{2}(\mbox{\rm diam}^p_p(S)\lambda)^{\frac{1}{p-1}-\frac{q}{p}}\epsilon^{q - \frac{1}{p-1}}.
\end{eqnarray*}
Since $\frac{1}{p} + \frac{1}{q}=1$, we have $\frac{1}{p-1}-\frac{q}{p}=0$ and $q - \frac{1}{p-1}=1$. These facts together with the definition of $\mbox{\rm diam}_p(S)$, as well as the inequality above, imply that
$$
\triangle L_z \le \frac{\epsilon}{2} + \frac{\epsilon}{2} =\epsilon,
$$
which, combined with statement (i), proves the desired result.
\end{proof}

\section{Algorithms and Their Iteration Complexities for Finding the $\epsilon$-Stationary Point}\label{Sec:de-alg}
We are now in a position to present our first algorithm for \eqref{prob:nonconvex}. In particular, at each iteration we find the search direction through optimizing a partially linearized function and then determine the step size by minimizing a simple one-dimensional function.

\begin{center}
\begin{tabular}{@{}llr@{}}\toprule
{\bf{Algorithm 1}} \\
\hline
{Let} $x^0 \in S$ {be given} and set $y^0 = x^0$.  \\
 \textbf{for} $k = 1,2,\cdots, N $, \textbf{do} \\
 \qquad $y^{k}=\arg\min_{y \in S} L(y;x^{k})$, and let $d^k = y^k - x^k$; \\
 \qquad $\alpha_k = \arg\min_{\alpha \in [0,1]}  \alpha\, \nabla f(x^{k})^{\top}d^k + \alpha^p\,\frac{\lambda}{2}\|d^k\|^p_p + (1 - \alpha)h(x^k) + \alpha\, h(y^k)$. \\
 \qquad Set $x^{k+1}=(1-\alpha_k)x^{k}+ \alpha_k y^{k}$. \\
\textbf{end for} \\
\hline
\end{tabular}
\end{center}

Note that in the absence of the nonsmooth part $h$, this algorithm is simply  CG (Conditional Gradient).
The computational complexity of {\bf Algorithm 1} to reach an $\epsilon$-stationary solution is as follows.
\begin{theorem}\label{theorm:CG} For any $0< \epsilon < \mbox{\rm diam}^p_p(S) \lambda $, {\bf Algorithm 1} finds an $\epsilon$-stationary point of~\eqref{prob:nonconvex} within $\left\lceil \frac{2(\Phi(x^{1}) - \Phi^*)(\mbox{\rm diam}^p_p(S)\lambda)^{q-1}}{\epsilon^q} \right\rceil$ steps, where $\frac{1}{p} + \frac{1}{q} =1$.
\end{theorem}
\begin{proof} According to Assumption~\ref{assump-p}, it holds that
\begin{eqnarray}\label{formu-in-proof}
&&-\nabla f(x^{k})^{\top}(x^{k+1}-x^{k}) - \frac{\lambda}{2}\|x^{k+1} - x^{k}\|^p_p+ h(x^{k})-h(x^{k+1}) \nonumber \\
&\le& f(x^{k}) - f(x^{k+1}) + h(x^{k})-h(x^{k+1}) \nonumber \\
&=& \Phi(x^{k}) - \Phi(x^{k+1}) .
\end{eqnarray}
Note that $\frac{\epsilon}{\mbox{\rm diam}^p_p(S) \lambda } \le 1$ and $x^{k+1} - x^{k} = \alpha_k(y^{k}-x^{k})$.
For simplicity, denote $\triangle L_k  := \triangle L_{x^k}$.
By the optimality of $\alpha_k$, we have
\begin{eqnarray*}
&&\left(\frac{\epsilon}{\mbox{\rm diam}^p_p(S)\lambda}\right)^{\frac{1}{p-1}}\triangle L_k - \frac{1}{2\lambda^{1/(p-1)}} \left(\frac{\epsilon}{ \mbox{\rm diam}_p(S)}\right)^{\frac{p}{p-1}} \\
&\le&-\left(\frac{\epsilon}{\mbox{\rm diam}^p_p(S)\lambda}\right)^{\frac{1}{p-1}}\left( \nabla f(x^{k})^{\top}(y^{k}-x^{k})+h(y^{k}) - h(x^{k}) \right)- \frac{\lambda}{2} \frac{\|y^{k} - x^{k}\|^p_p}{\mbox{\rm diam}^p_p(S)}\left(\frac{\epsilon}{\lambda \mbox{\rm diam}_p(S)}\right)^{\frac{p}{p-1}}\\
&=&-\left(\frac{\epsilon}{\mbox{\rm diam}^p_p(S)\lambda}\right)^{\frac{1}{p-1}}\left( \nabla f(x^{k})^{\top}(y^{k}-x^{k})+h(y^{k}) - h(x^{k}) \right)- \frac{\lambda}{2} \left(\frac{\epsilon}{\mbox{\rm diam}^p_p(S)\lambda}\right)^{\frac{p}{p-1}} \|y^{k} - x^{k}\|^p_p \\
&\le&-\alpha_k \left( \nabla f(x^{k})^{\top}(y^{k}-x^{k})+h(y^{k}) - h(x^{k}) \right) - \frac{\lambda \alpha_k^p}{2}\|y^{k} - x^{k}\|^p_p\\
&=& -\nabla f(x^{k})^{\top}(\alpha_k (y^{k}-x^{k})) + h(x^{k})-(1 - \alpha_k)h(x^k) - \alpha_k h(y^k) - \frac{\lambda }{2}\|\alpha_k(y^{k} - x^{k})\|^p_p \\
&\le & -\nabla f(x^{k})^{\top}(x^{k+1}-x^{k}) + h(x^{k})-h(x^{k+1}) - \frac{\lambda}{2}\|x^{k+1} - x^{k}\|^p_p,
\end{eqnarray*}
where the last inequality is due to the convexity of function $h(\cdot)$. Combining this formula with~\eqref{formu-in-proof} leads to
$$
\left(\frac{\epsilon}{\mbox{\rm diam}^p_p(S)\lambda}\right)^{\frac{1}{p-1}} \triangle L_k \le \Phi(x^{k}) - \Phi(x^{k+1}) + \frac{1}{2\lambda^{1/(p-1)}} \left(\frac{\epsilon}{ \mbox{\rm diam}_p(S)}\right)^{\frac{p}{p-1}}.
$$
Dividing both sides by $\left(\frac{\epsilon}{\mbox{\rm diam}^p_p(S)\lambda}\right)^{\frac{1}{p-1}}$, one has
$$
\triangle L_k \le \left(\frac{\epsilon}{\mbox{\rm diam}^p_p(S)\lambda}\right)^{-\frac{1}{p-1}}\left(\Phi(x^{k}) - \Phi(x^{k+1})\right) + \frac{\epsilon}{2}.
$$
Summing up the above inequalities for $k=1,\ldots,N$ yields
\begin{eqnarray*}
N \min_{k\in\{1,\ldots,N\}}\triangle L_k  &\le&  \sum_{k=1}^{N}\triangle L_k  \le \left(\frac{\epsilon}{\mbox{\rm diam}^p_p(S)\lambda}\right)^{-\frac{1}{p-1}} \left(\Phi(x^{1}) - \Phi(x^{N+1})\right)+ \frac{\epsilon}{2}N \\
&\le& \left(\frac{\epsilon}{\mbox{\rm diam}^p_p(S)\lambda}\right)^{-\frac{1}{p-1}} \left(\Phi(x^{1}) - \Phi^*\right)+ \frac{\epsilon}{2}N .
\end{eqnarray*}
Observe that $\frac{1}{p} + \frac{1}{q}=1$ leads to $q - 1 - \frac{1}{p-1}=0$.
Dividing the above inequalities by $N = \left\lceil \frac{2(\Phi(x^{1}) - \Phi^*)(\mbox{\rm diam}^p_p(S)\lambda)^{q-1}}{\epsilon^q} \right\rceil$, we arrive at the conclusion that there must exist some $\tilde{k} \le N$ such that
$\triangle L_{\tilde{k}} \le \epsilon$,
which combined with Lemma~\ref{kkt-criteria} further implies that $x^{\tilde{k}}$ is an $\epsilon$-stationary point of~\eqref{prob:nonconvex}.
\end{proof}

From Theorem~\ref{theorm:CG}, we can see that the larger value of $p$ implies fewer iteration numbers required by the algorithm.
Since a concave function satisfies {Assumption}~\ref{assump-p} for any $p \ge 1$, in this case a better 
complexity bound is guaranteed.
\begin{corollary}\label{coro:concave}
If $f$ is a concave function, then performing {\bf Algorithm 1} with $\alpha_k = 1$ for all $k$ will reach an $\epsilon$-stationary point of~\eqref{prob:nonconvex} within $\left\lceil \frac{\Phi(x^{1}) - \Phi^*}{\epsilon} \right\rceil$ steps.
\end{corollary}

To proceed, we present our second algorithm below, which is based on minimizing an upper bound of the original objective function at each iteration.
\begin{center}
\begin{tabular}{@{}llr@{}}\toprule
{\bf{Algorithm 2}}\\
\hline
{Let} $x^1 \in S$ {be given}  \\
 \textbf{for} $k = 1,2,\cdots, N $, \textbf{do}\\
 \qquad $x^{k+1}=\arg\min_{y \in S} U(y;x^{k})$.\\
\textbf{end for}\\
\hline
\end{tabular}
\end{center}

We remark that this algorithm still solves a separable convex subproblem since the function $U(y;x^{k})$ itself is convex when $p\ge 1$ and is separable with respect to $y_i$ for all $i$.
For simplicity, we denote $\triangle U_k := \triangle U_{x^k,p}$.
Below is the complexity result for {\bf Algorithm 2}.
\begin{theorem}\label{theorm:PPA} For any any $0< \epsilon < \mbox{\rm diam}^p_p(S) \lambda $, {\bf Algorithm 2} finds an $\epsilon$-stationary point of~\eqref{prob:nonconvex} within $\left\lceil \frac{2(\Phi(x^{1}) - \Phi^*)(\mbox{\rm diam}^p_p(S)\lambda)^{q-1}}{\epsilon^q} \right\rceil$ steps, where $\frac{1}{p} + \frac{1}{q} =1$.
\end{theorem}
\begin{proof} According to Assumption~\ref{assump-p}, it holds that
$$
f(x^{k+1}) \le f(x^{k}) + \nabla f(x^{k})^{\top}(x^{k+1}-x^{k}) + \frac{\lambda}{2}\|x^{k+1} - x^{k}\|^p_p.
$$
As a result,
\begin{eqnarray*}
\triangle U_k &=& -\nabla f(x^{k})^{\top}(x^{k+1}-x^{k}) - \frac{\lambda}{2}\|x^{k+1} - x^{k}\|^p_p + h(x^{k}) - h(x^{k+1})\\
 &\le & -f(x^{k+1}) + f(x^{k}) + h(x^{k}) - h(x^{k+1})\\
&=& \Phi(x^{k}) - \Phi(x^{k+1}).
\end{eqnarray*}
Summing up the above inequalities for $k=1,\ldots,N$ yields
$$
N \min_{k\in\{1,\ldots,N\}}\triangle U_k \le \sum_{k=1}^{N}\triangle U_k \le \Phi(x^{1}) - \Phi(x^{N+1}) \le \Phi(x^{1}) - \Phi^*.
$$
Thus, setting $N = \left\lceil \frac{2(\Phi(x^{1}) - \Phi^*)(\mbox{\rm diam}^p_p(S)\lambda)^{q-1}}{\epsilon^q} \right\rceil$, there must exist some $\tilde{k} \le N$ such that
$$
\triangle U_{\tilde{k}} \le \frac{1}{2(\mbox{\rm diam}^p_p(S)\lambda)^{q-1}}\epsilon^q=\frac{1}{2}\left(\frac{\epsilon}{\mbox{\rm diam}_p(S)\lambda^{1/p}} \right)^q.
$$
This inequality combined with statement (ii) in Lemma~\ref{kkt-criteria} implies that $x^{\tilde{k}}$ is an $\epsilon$-stationary point for~\eqref{prob:nonconvex}.
\end{proof}

\section{Iteration Complexity Bounds for Stochastic and Smoothing Approximation Methods}\label{Sec:sto-alg}
\subsection{Complexity for Stochastic Approximation}
In this subsection, we study the case where the exact gradient of $f(x)$ in problem~\eqref{prob:nonconvex} is not available. Instead, we assume that a noise estimation for the gradient of $f$ can be obtained by resorting to the so-called stochastic first-order oracle $\mathcal{SFO}$. In particular, for point $x^k \in S$ at the $k$-th iteration, $\mathcal{SFO}$ would return a stochastic gradient $G(x^k, \xi^k)$, where $\xi^k$ is a random variable satisfying
\begin{eqnarray}
& &\ex[G(x^k, \xi^k)] = \nabla f(x^k), \label{stochastic-assump-1}\\
& &\ex[\| G(x^k, \xi^k) - \nabla f(x^k) \|_q^q ] \le \sigma^q, \label{stochastic-assump-2}
\end{eqnarray}
for some constant $\sigma > 0$.

The method of randomized sampling for stochastic programming can be traced back to the seminal paper of Robbins and Monro~\cite{RobbinsMonro1951} (1951).
Computational complexity for convex optimization was first studied in~\cite{NemirovskiYudin83}; similar results were subsequently established for convex stochastic optimization~\cite{NemirovskiJuditskyLanShapiro09,Lan12}. Recently, complexity analysis for stochastic approximation has been successfully extended to several nonconvex models as well; cf.~\cite{GhadimiLan13,GhadimiLanZhang13,WangMaYuan13}.

Inspired by~\cite{GhadimiLanZhang13}, we propose below a mini-batch stochastic algorithm for the stochastic version of~\eqref{prob:nonconvex}:

\begin{center}
\begin{tabular}{@{}llr@{}}\toprule
{\bf{Algorithm 3}}\\
\hline
{Let} $x^1 \in S$ {be given}  \\
 \textbf{for} $k = 1,2,\cdots, N $, \textbf{do}\\
 \qquad Call the $\mathcal{SFO}$ $m_k$ times to obtain $G(x^k,\xi^{k,i})$, $i=1,\ldots,m_k$; \\
 \qquad set $G_k = \frac{1}{m_k}\sum_{i=1}^{m_k}G(x^k,\xi^{k,i})$, and compute\\
 \qquad $x^{k+1}=\arg\min\limits_{x \in S}\tilde{U}(x;G_k)$, where
  $\tilde{U}(x;G_k):=G_k^{\top}(x-x^k) + \frac{\lambda}{2}\|x-x^k \|^p_p + h(x)$.\\
\textbf{end for}\\
\hline
\end{tabular}
\end{center}

Before discussing the computational complexity of {\bf Algorithm 3}, we shall note the following two technical lemmas.

\begin{lemma}\label{lemma:p-norm} For any $a,b,c \in \RR$ and 
$p \ge 2$ we have
\begin{equation} \label{ineq1}
(\sign(a - c)|a-c|^{p-1}- \sign(b - c)|b-c|^{p-1})(a-b) \ge (1/2)^{p-2}|a-b|^{p}.
\end{equation}
\end{lemma}
\begin{proof}
First of all, we observe that $x^{p-1}$ is a convex function for $x\in \RR_+$, and so for any $x,y\in \RR_+$ we have
\[
\left( \frac{x+y}{2}\right)^{p-1} \le \frac{1}{2}\left( x^{p-1}+y^{p-1} \right)
\]
implying that $x^{p-1}+y^{p-1} \ge \frac{1}{2^{p-2}}(x+y)^{p-1}$. Also, we have $(x+y)^{p-1}\ge x^{p-1}+y^{p-1}$, because $(x^{p-1}+y^{p-1})^{1/(p-1)}$ is the $L_{p-1}$-norm of $(x,y)$ which can never exceed its corresponding $L_1$-norm.

To prove the lemma, due to symmetry we need only to consider three separate cases: (i) $c<b<a$; (ii) $c<a<b$; (iii) $a<c<b$.

In case (i), the LHS of \eqref{ineq1} equals
\[
\left[ (a-c)^{p-1} - (b-c)^{p-1} \right] (a-b) \ge (a-c)^{p-1}(a-b) \ge (a-b)^p \ge \frac{1}{2^{p-2}} (a-b)^p.
\]
In case (ii), the LHS of \eqref{ineq1} equals
\[
\left[ (b-c)^{p-1} - (a-c)^{p-1}\right] (b-a) \ge (b-a)^{p-1}(b-a) \ge \frac{1}{2^{p-2}} (a-b)^p.
\]
Finally, in case (iii), the LHS of \eqref{ineq1} equals
\[
\left[ (c-a)^{p-1} + (b-c)^{p-1} \right] (b-a) \ge \frac{1}{2^{p-2}} (b-a)^{p-1}(b-a) = \frac{1}{2^{p-2}} (b-a)^p.
\]
Summarizing all the cases, the claimed inequality \eqref{ineq1} follows.
\end{proof}
\begin{lemma}\label{lemma:pertrb} For integer $p\ge 2$, let 
$$
x^*_i = \min\limits_{x \in S} g_i^{\top}(x-z) + \frac{\lambda}{2}\|x-z \|^p_p + h(x),\;\,\mbox{for}\;\,i=1,2.
$$
Then it holds that
$$
(1/2)^{p}\| x^*_1 -  x^*_2\|^p_p \le 1/(\lambda\,p)^q\|g_1 - g_2 \|_q^q,\;\,\mbox{with}\;\, 1/p + 1/q = 1.
$$
\end{lemma}
\begin{proof} By the optimality of $x^*_1 $ and $x^*_2$, for any $x \in S$, there exist $w_1 \in \partial h(x^*_1)$ and $w_2 \in \partial h(x^*_2)$ such that
\begin{equation}\label{prop-opt-g1}
\left( g_1 +  \frac{\lambda}{2}p\,\sign(x^*_1 - z)\|x^*_1 - z\|_{p-1}^{p-1} + w_1 \right)^{\top}(x - x^*_1) \ge 0,
\end{equation}
and
\begin{equation}\label{prop-opt-g2}
\left( g_2 +  \frac{\lambda}{2}p\,\sign(x^*_2 - z)\|x^*_2 - z\|_{p-1}^{p-1} + w_2 \right)^{\top}(x - x^*_2) \ge 0.
\end{equation}
Letting $x = x^*_2$ in~\eqref{prop-opt-g1} and $x = x^*_1$ in~\eqref{prop-opt-g2}, by Lemma~\ref{lemma:p-norm} and summing up~\eqref{prop-opt-g1} and~\eqref{prop-opt-g2}, we have
\begin{eqnarray*}
&&(g_1 - g_2)^{\top}(x^*_2 - x^*_1)\\
&\ge& \frac{\lambda}{2}p\left(\sign(x^*_1 - z)\|x^*_1 - z\|_{p-1}^{p-1} - \sign(x^*_2 - z)\|x^*_2 - z\|_{p-1}^{p-1}  \right)^{\top}(x^*_1 - x^*_2) +(w_1 - w_2)^{\top}(x^*_1 - x^*_2)\\
&\ge&\frac{\lambda\,p}{2^{p-1}}\|x^*_1 - x^*_2\|_p^p+ (w_1 - w_2)^{\top}(x^*_1 - x^*_2) \\
&\ge&\frac{\lambda\,p}{2^{p-1}}\|x^*_1 - x^*_2\|_p^p,
\end{eqnarray*}
where the last inequality is due to the convexity of $h$. On the other hand, by the H\"{o}lder inequality one has
$$
(g_1 - g_2)^{\top}(x^*_2 - x^*_1) \le \|g_1 - g_2\|_q\|x^*_1 - x^*_2\|_p\;\,\mbox{with}\;\, 1/p + 1/q = 1.
$$
The desired result follows by combining these two inequalities and then taking the $q$-th power on both sides.
\end{proof}

Recall that $\tilde{U}(x;G_k)=G_k^{\top}(x-x^k) + \frac{\lambda}{2}\|x-x^k \|^p_p + h(x)$. We are ready to present the main result of this subsection.
\begin{theorem} \label{complexity-stochastic}
Suppose $\{x^k \}$ is the sequence of iterates generated by {\bf Algorithm 3}. Denote $$\triangle \tilde{U}_k = \tilde{U}(x^k;G_k)-\tilde{U}(x^{k+1};G_k)\;\, \mbox{and}\;\,
\tilde{k}=\arg \min_{k\in\{1,\ldots,N\}}\triangle \tilde{U}_k.$$
Then, we have
\begin{equation}\label{stochastic-bound}
\ex [\triangle {U}_{\tilde{k}} ] \le \ex [\triangle \tilde{U}_{\tilde{k}} ] \le \left(\frac{2\,\sigma ^{q}}{(\lambda\,p)^{q/p}}\sum_{k=1}^{N}\frac{1}{m_k^{q-1}}  + \Phi(x^{1}) - \Phi^*\right)\bigg{/}N,
\end{equation}
where $p \ge 2$ and $\frac{1}{p} + \frac{1}{q} =1$. Moreover, if we assume that the batch sizes $m_k = m$ for $k=1,\ldots,N$ with some $m \ge 1$, then
\begin{equation}\label{stochastic-bound2}
\ex [\triangle {U}_{\tilde{k}} ] \le \ex [\triangle \tilde{U}_{\tilde{k}} ] \le  \frac{2\,\sigma ^{q}}{(\lambda\,p)^{q/p}}\frac{1}{m^{q-1}}  + \frac{\Phi(x^{1}) - \Phi^*}{N}.
\end{equation}
\end{theorem}
\begin{proof} Let $\delta_k = \nabla f(x^k) - G_k$ and denote
\begin{equation}\label{exact-uk}
y^{k+1} = \arg \min_{y \in S} U(y;x^k)= \arg \min_{y \in S} \nabla f(x^k)^{\top}(y-x^k) + \frac{\lambda}{2}\|y-x^k \|^p_p + h(y).
\end{equation}
Then it follows from Assumption~\ref{assump-p} that
\begin{eqnarray*}
\triangle \tilde{U}_k  & = & - G_k^{\top}(x^{k+1}-x^k) - \frac{\lambda}{2}\|x^{k+1}-x^k \|^p_p + h(x^k) - h(x^{k+1})  \\
&=& \delta_k^{\top} (x^{k+1}-x^k) - \nabla f(x^k)^{\top}(x^{k+1}-x^k) - \frac{\lambda}{2}\|x^{k+1}-x^k \|^p_p + h(x^k) - h(x^{k+1})  \\
&\le& \delta_k^{\top} (y^{k+1}-x^k) + \delta_k^{\top} (x^{k+1}-y^{k+1}) + f(x^k) - f(x^{k+1}) + h(x^k) - h(x^{k+1})  \\
&\le & \delta_k^{\top} (y^{k+1}-x^k) + \|\delta_k\|_q\|x^{k+1}-y^{k+1}\|_p + \Phi(x^k) - \Phi(x^{k+1})\\
&\le & \delta_k^{\top} (y^{k+1}-x^k) + \frac{2}{(\lambda\,p)^{q/p}}\| \delta_k \|_q^{q}+ \Phi(x^k) - \Phi(x^{k+1}),
\end{eqnarray*}
where the last inequality is due to Lemma~\ref{lemma:pertrb} and the fact $\frac{1}{p} + \frac{1}{q} = 1$.
Now, summing up the above inequalities for $k=1,\ldots, N$, we obtain
\begin{equation}\label{approx-sum-uk}
\sum_{k=1}^{N}\triangle \tilde{U}_k  \le \sum_{k=1}^{N}\delta_k^{\top} (y^{k+1}-x^k) + \frac{2}{(\lambda\,p)^{q/p}}\sum_{k=1}^{N}\| \delta_k \|_q^{q} + \Phi(x^{1}) - \Phi^*.
\end{equation}
Let $\xi_{[k-1]}$ be the random samples generated before iteration $k$. So, at the time $x^k$ and $y^{k+1}$ were determined, $\xi_{[k-1]}$ was already realized. Consequently, by Assumption~\eqref{stochastic-assump-1}, we have
\begin{equation}\label{conditional-mean0}\ex\left[\delta_k^{\top} (y^{k+1}-x^k)\, \big{|} \, \xi_{[k-1]} \right]=0.
\end{equation}
In addition, according to Assumption~\eqref{stochastic-assump-2}, one has
$$
\ex[ \| \delta_k \|_q^{q} ] = \frac{1}{m_k^q} \sum_{i=1}^{m_k}\ex [\|\nabla f(x^k) - G(x^k,\xi_{k,i})\|_q^q] \le \frac{1}{m_k^{q-1}} \sigma ^{q}.
$$
The above two formulas and~\eqref{approx-sum-uk} lead to
\begin{equation}\label{N-min}
N \, \ex[\triangle \tilde{U}_{\tilde{k}}] \le \sum_{k=1}^{N} \ex[ \triangle \tilde{U}_k ] \le \frac{2\,\sigma ^{q}}{(\lambda\,p)^{q/p}}\sum_{k=1}^{N}\frac{1}{m_k^{q-1}}  + \Phi(x^{1}) - \Phi^*.
\end{equation}
Furthermore, by definition of $y^{k+1}$ in~\eqref{exact-uk}, one has
\begin{eqnarray*}
\triangle {U}_k  & = & - \nabla f(x^k)^{\top}(y^{k+1}-x^k) - \frac{\lambda}{2}\|y^{k+1}-x^k \|^p_p + h(x^k) - h(y^{k+1})  \\
&=& -\delta_k^{\top} (y^{k+1}-x^k) - G_k^{\top}(y^{k+1}-x^k) - \frac{\lambda}{2}\|y^{k+1}-x^k \|^p_p + h(x^k) - h(y^{k+1})  \\
&\le& -\delta_k^{\top} (y^{k+1}-x^k) + \triangle \tilde{U}_k ,
\end{eqnarray*}
which in combination with~\eqref{conditional-mean0} and~\eqref{N-min} yields~\eqref{stochastic-bound}.
\end{proof}

In order to apply Part (ii) of Lemma~\ref{kkt-criteria}, we want $\ex [\triangle {U}_{\tilde{k}} ]$ to be upper bounded by $\frac{1}{2}\left(\frac{\epsilon}{\mbox{\rm diam}_p(S)\lambda^{1/p}}\right)^q$. Since the bound in the above theorem depends on $m$, we present the following corollary to show how $m$ can be chosen so as to achieve the sharpest result. 
\begin{corollary} Suppose $\triangle {U}_{\tilde{k}}$ is defined in Theorem~\ref{complexity-stochastic}, integer $p \ge 2$ and $\frac{1}{p} + \frac{1}{q} = 1$. For a given sufficiently small $\epsilon$, if the total number of calls $\bar{N}$ to the $\mathcal{SFO}$ is given by
\begin{equation}\label{choose-N}
\bar{N}= \left\lceil \frac{(4\sigma)^p\lambda^{q-1}\mbox{\rm diam}_p(S)^{pq}(\Phi(x^1)-\Phi^*)}{p\,\epsilon^{pq}}\left(1+(q-1)^{-1/p+1/q} \right)^p \right\rceil,
\end{equation}
and in each iteration of {\bf Algorithm 3} the batch size is set to be
$$
m = \left\lceil \min\left\{ \max\left\{1 , \frac{\sigma((q-1)\bar{N})^{1/q}}{(\lambda\,p)^{1/p}(\Phi(x^1)-\Phi^*)^{1/q}} \right\},\bar{N} \right\} \right\rceil,
$$
then we have
$$
\ex [\triangle {U}_{\tilde{k}} ] \le \frac{1}{2}\left(\frac{\epsilon}{\mbox{\rm diam}_p(S)\lambda^{1/p}}\right)^q.
$$
\end{corollary}
\begin{proof} We first assume that 
\begin{equation}\label{N-large}
\bar{N} \ge \frac{\sigma^p(q-1)^{p/q}}{\lambda\,p(\Phi(x^1)-\Phi^*)^{p/q}}
\end{equation}
or equivalently, $\bar{N} \ge \frac{\sigma((q-1)\bar{N})^{1/q}}{(\lambda\,p)^{1/p}(\Phi(x^1)-\Phi^*)^{1/q}}$.
Since the batch size $m_k$ at each iteration is identical, {\bf Algorithm 3} can perform at most $N = \lfloor \bar{N}/m \rfloor$ iterations, which implies $N \ge \bar{N}/(2m)$. This fact together with~\eqref{stochastic-bound2} yields
\begin{eqnarray*}
\ex [\triangle {U}_{\tilde{k}} ] &\le& \frac{2\,\sigma ^{q}}{(\lambda\,p)^{q/p}}\frac{1}{m^{q-1}}  + \frac{2m(\Phi(x^{1}) - \Phi^*)}{\bar{N}}\\
&\le&  \frac{2\sigma(\Phi(x^1)-\Phi^*)^{1/p}}{(\lambda\,p\,(q-1)\,\bar{N})^{1/p}} + \frac{2(\Phi(x^{1}) - \Phi^*)}{\bar{N}}\left(1+ \frac{\sigma((q-1)\bar{N})^{1/q}}{(\lambda\,p)^{1/p}(\Phi(x^1)-\Phi^*)^{1/q}} \right)\\
&=& \frac{2(\Phi(x^{1}) - \Phi^*)}{\bar{N}} + \frac{2\sigma(\Phi(x^1)-\Phi^*)^{1/p}}{(\lambda\,p\,\bar{N})^{1/p}}\left((q-1)^{-1/p+1/q} \right),
\end{eqnarray*}
where 
we used the fact $\frac{1}{p} + \frac{1}{q}=1$.
If we further assume
\begin{equation}\label{N-large2}
\bar{N} \ge \frac{(\lambda\,p)^{q/p}\sigma^q}{\Phi(x^1)-\Phi^*} ,
\end{equation}
then
$$
\ex [\triangle {U}_{\tilde{k}} ] \le \frac{2\sigma(\Phi(x^1)-\Phi^*)^{1/p}}{(\lambda\,p\,\bar{N})^{1/p}}\left(1+(q-1)^{-1/p+1/q} \right).
$$
Thus, by choosing $\bar{N}$ according to~\eqref{choose-N}, when $\epsilon$ is sufficiently small so that~\eqref{N-large} and~\eqref{N-large2} are satisfied, then the desired result follows.

\end{proof}

\subsection{Complexity for Smoothing Approximation}
In this subsection, we shall consider a modified model where the nonsmooth part $h(x)$ is assumed to be concave rather than convex, while the smooth part $f(x)$ is still assumed to satisfy Assumption~\ref{assump-p}. In fact, this model is frequently encountered in several applications. For instance, in the $L_2$-$L_q$ minimization problem~\cite{ChenGeWangYe14}, Assumption~\ref{assump-p} holds with $p=2$ since $f(x)=\|Ax-b\|_2^2$, and $h(x)=\|x\|_q^q$ with $0 < q <1$ is a concave function. In this case, we apply the smoothing approximation to the nonsmooth function $h(x)$. In particular, consider the following convolution between $h$ and $\mu$:
$$
h_\mu(x): = \int_{\RR^n}h(x+y)\mu(y)dy = \ex_\mu [h(x+Z)],
$$
where $Z$ is a random variable with probability density $\mu$. It is well known (cf.~\cite{Bertsekas73}) that if $\mu$ is a density with respect to Lebesgue measure, then $h_\mu$ is differentiable.
Furthermore, it is well known that (see, e.g.,~\cite{DuchiBartlettWainwright12}) $h_\mu(x)$ and $h(x)$ can be bounded from each other using the properties of $\mu$, as stated in the following lemma.
\begin{lemma}\label{smooth-property}
Let $\xi$ be a random variable with the uniform density $\mu$ over the $L_2$-ball (radius $r$). Assume that $\sup\{\|g \|_2^2 \; | \; g \in \partial h(x)\} \le M$ for $x \in S+\mathcal{B}_2(0,r),$  where $\mathcal{B}_2(0,r)=\{ y\; | \; \|y\|_2 \le r \}$ is the Euclidean ball with radius $r$.
Let
\begin{equation}
h_r(x):= \ex_\mu[h(x+r\xi) ].
\end{equation}
It holds that \\
(i) If $h(x)$ is concave then $h_r(x)$ is a concave function as well;\\
(ii) $h(x)\le h_{r}(x) \le h(x) + M r$;\\
(iii) If $h(x)$ is differentiable, then $\ex [\nabla h(x+r\xi)]=\nabla h_r(x)$ and $\ex [\|\nabla h(x+r\,\xi) - \nabla h_r(x)\|_2^2] \le M^2$.
\end{lemma}

Now we consider the problem
\begin{eqnarray}\label{prob:smoothing}\begin{array}{lll} \min & \Phi_r(x):=f(x) + h_r(x) \\ \st & x \in S \subseteq\RR^n,&  \end{array}\end{eqnarray}
Denote $x^*$ and $\tilde{x}$ to be optimal solutions of~\eqref{prob:nonconvex} and~\eqref{prob:smoothing} respectively. Then from Lemma~\ref{smooth-property} and the optimality of $\tilde{x}$, we have
\begin{equation}\label{perturbation-bound}
f(\tilde{x}) + h(\tilde{x}) \le f(\tilde{x}) + h_r(\tilde{x}) \le f({x^*}) + h_r({x^*}) \le f({x^*}) + h({x^*}) + Mr,
\end{equation}
which means that if the perturbation is small, then the smoothing version~\eqref{prob:smoothing} is indeed a good approximation for the original problem~\eqref{prob:nonconvex}. Denote
$${L}_r(y;x^k)=  (\nabla f(x^{k}) + \nabla h_r(x^k))^{\top}(x-x^{k}),\; \mbox{and} \;\triangle {L}_k = {L}_r(x^k;x^k) - {L}_r(z^{k};x^k),$$
where $z^{k} = \min_{z \in S} L_r(z;x^k)$.
According to Lemma~\ref{kkt-criteria}, $x^k$ is an $\epsilon$-stationary point of~\eqref{prob:smoothing}, if $\tilde{L}_k \le \epsilon$.
We now propose a sampling-smoothing algorithm for~\eqref{prob:smoothing} as follows:
\begin{center}
\begin{tabular}{@{}llr@{}}\toprule
{\bf{Algorithm 4}}\\
\hline
{Let} $x^0 \in S$ {be given} and set $y^0 = x^0$.  \\
 \textbf{for} $k = 1,2,\cdots, N $, \textbf{do}\\
 \qquad Draw i.i.d.\ random samples $(\xi^{k,1},\ldots,\xi^{k,m_k})$, and set $G_k = \frac{1}{m_k}\sum_{i=1}^{m_k}\nabla h(x^k+r\xi^{k,i})$.\\
 \qquad Compute $y^{k}=\arg\min_{y \in S} \tilde{L}(y;x^{k})$, where $\tilde{L}(y;G_k)=  (\nabla f(x^{k}) + G_k)^{\top}(x-x^{k})$.\\
 \qquad Let $d^k = y^k - x^k$, and $\alpha_k = \arg\min_{\alpha \in [0,1]}  \alpha\, (\nabla f(x^{k})+ G_k))^{\top}d^k + \alpha^p\,\frac{\lambda}{2}\|d^k\|^p_p$.\\
 \qquad Set $x^{k+1}=(1-\alpha_k)x^{k}+ \alpha_k y^{k}$.\\
\textbf{end for}\\
\hline
\end{tabular}
\end{center}

The iteration complexity of this algorithm is presented in the following theorem.
\begin{theorem}\label{complexity-smoothing}Suppose $\{x^k \}$ is the sequence generated by {\bf Algorithm 4}. Denote $$\triangle \tilde{L}_k = \tilde{L}(x^k;G_k) - \tilde{L}(y^{k};G_k) \;\, \mbox{and}\;\,
\tilde{k}=\arg \min_{k\in\{1,\ldots,N\}}\triangle \tilde{L}_k.$$
For $\epsilon \le \mbox{\rm diam}^p_p(S) \lambda $, let $m_k =  \left\lceil \frac{ \mbox{\rm diam}_2^2(S)M^2 N^2}{ (\Phi(x^1)-\Phi^*)^{2}} \right\rceil$ for $k=1,\ldots,N$ and $N = \left\lceil \frac{4(\Phi(x^{1}) - \Phi^*)(\mbox{\rm diam}^p_p(S)\lambda)^{q-1}}{\epsilon^q} \right\rceil$;
then we have
\begin{equation}\label{stochastic-bound3}
\ex [\triangle {L}_{\tilde{k}} ] \le \ex [\triangle \tilde{L}_{\tilde{k}} ] \le \epsilon,
\end{equation}
where $\frac{1}{p} + \frac{1}{q} =1$.
\end{theorem}
\begin{proof} Denote $\delta_k = \nabla h_r(x^k) - G_k$.
By Assumption~\ref{assump-p} and the optimality of $\alpha_k$, and noting that $\frac{\epsilon}{\mbox{\rm diam}^p_p(S) \lambda } \le 1$ and $x^{k+1} - x^{k} = \alpha_k(y^{k}-x^{k})$, we obtain the following sequence of inequalities
\begin{eqnarray*}
&&\left(\frac{\epsilon}{\mbox{\rm diam}^p_p(S)\lambda}\right)^{\frac{1}{p-1}} \triangle \tilde{L}_k - \delta_k^{\top}(x^{k+1}-x^{k}) - \frac{\lambda}{2} \left(\frac{\epsilon}{\lambda \mbox{\rm diam}_p(S)}\right)^{\frac{p}{p-1}} \\
&\le&-\left(\frac{\epsilon}{\mbox{\rm diam}^p_p(S)\lambda}\right)^{\frac{1}{p-1}}  (\nabla f(x^{k})+G_k)^{\top}(y^{k}-x^{k})- \frac{\lambda}{2} \frac{\|y^{k} - x^{k}\|^p_p}{\mbox{\rm diam}^p_p(S)}\left(\frac{\epsilon}{\lambda \mbox{\rm diam}_p(S)}\right)^{\frac{p}{p-1}}- \delta_k^{\top}(x^{k+1}-x^{k})  \\
&\le &-\alpha_k  (\nabla f(x^{k})+G_k)^{\top}(y^{k}-x^{k}) - \frac{\lambda \alpha_k^p}{2}\|y^{k} - x^{k}\|^p_p- \delta_k^{\top}(x^{k+1}-x^{k}) \\
&=& - (\nabla f(x^{k})+G_k)^{\top}(\alpha_k (y^{k}-x^{k})) - \frac{\lambda }{2}\|\alpha_k(y^{k} - x^{k})\|^p_p - \delta_k^{\top}(x^{k+1}-x^{k}) \\
&= & - (\nabla f(x^{k})+G_k)^{\top}(x^{k+1}-x^{k}) - \frac{\lambda}{2}\|x^{k+1} - x^{k}\|^p_p- \delta_k^{\top}(x^{k+1}-x^{k}) \\
&= & - (\nabla f(x^{k})+\nabla h_r(x^{k}))^{\top}(x^{k+1}-x^{k}) - \frac{\lambda}{2}\|x^{k+1} - x^{k}\|^p_p \\
&\le & f(x^{k}) - f(x^{k+1}) + h_r(x^{k})-h_r(x^{k+1})   \\
& = & \Phi_r(x^{k}) - \Phi_r(x^{k+1}),
\end{eqnarray*}
where the last inequality follows from Assumption~\ref{assump-p} and concavity of $h_r(\cdot)$. Dividing both sides of the above inequality by $\left(\frac{\epsilon}{\mbox{\rm diam}^p_p(S)\lambda}\right)^{\frac{1}{p-1}}$ and rearranging the terms yield
$$
\triangle \tilde{L}_k \le \left(\frac{\epsilon}{\mbox{\rm diam}^p_p(S)\lambda}\right)^{-\frac{1}{p-1}}\left(\Phi(x^{k}) - \Phi(x^{k+1})+ \|\delta_k\|_2\| x^{k+1} - x^k\|_2 \right)  + \frac{\epsilon}{2}.
$$
Since $h(\cdot)$ is concave and $S$ is compact, $\bigcup_{x \in S + \mathcal{B}_2(0,r)}\partial h(x)$ is bounded. According to Lemma~\ref{smooth-property},
$$
(\ex [\| \delta_k \|_2])^2\le \ex [\| \delta_k \|_2^2] = \frac{1}{m_k^2}\sum_{i=1}^{m_k}\ex [\|\nabla h(x^k+r\xi^{k,i})- \nabla h_r(x) \|^2_2 ] \le \frac{M^2}{m_k}.
$$
Therefore, summing over $k=1,\ldots,N$ and taking expectation, one has
\begin{eqnarray*}
N \ex [ \triangle \tilde{L}_{\tilde k}] & \le &  \sum_{k=1}^{N}\ex [ \triangle \tilde{L}_k ] \\ &\le& \left(\frac{\epsilon}{\mbox{\rm diam}^p_p(S)\lambda}\right)^{-\frac{1}{p-1}} \left(\Phi(x^{1}) - \Phi(x^{N+1})+ \sum_{k=1}^{N}\frac{M}{\sqrt{m_k}} \|x^{k+1} - x^k\|_2 \right)+ \frac{\epsilon}{2}N \\
&\le& \left(\frac{\epsilon}{\mbox{\rm diam}^p_p(S)\lambda}\right)^{-\frac{1}{p-1}} \left(\Phi(x^{1}) - \Phi^* + \mbox{\rm diam}_2(S) \sum_{k=1}^{N}\frac{M}{\sqrt{m_k}} \right)+ \frac{\epsilon}{2}N .
\end{eqnarray*}
When $m_k$ and $N$ are chosen as described, we have
\begin{eqnarray*}
\ex [ \triangle \tilde{L}_{\tilde k} ]
&\le& \left(\frac{\epsilon}{\mbox{\rm diam}^p_p(S)\lambda}\right)^{-\frac{1}{p-1}} \left( \frac{\Phi(x^{1}) - \Phi^*}{N}+  \frac{\mbox{\rm diam}_2(S) M}{\sqrt{m_1}} \right)+ \frac{\epsilon}{2} \\
&\le& \left(\frac{\epsilon}{\mbox{\rm diam}^p_p(S)\lambda}\right)^{-\frac{1}{p-1}} \left( \frac{\Phi(x^{1}) - \Phi^*}{N}+  \frac{\Phi(x^{1}) - \Phi^*}{N} \right) + \frac{\epsilon}{2}\\
&\le& \frac{\epsilon}{4} + \frac{\epsilon}{4} + \frac{\epsilon}{2} = \epsilon.
\end{eqnarray*}

Since $z^{k} = \min_{z \in S} L_r(z;x^k)$,
then
\begin{eqnarray*}
\triangle {L}_{\tilde k} &=&  - (\nabla f(x^k)+ \nabla h_r(x^k) )^{\top}(z^k-x^k) \\
&=&- (\nabla f(x^k)+ G_k )^{\top}(z^k-x^k) -\delta_k^{\top}(z^k-x^k)\\
&\le & - (\nabla f(x^k)+ G_k )^{\top}(y^k-x^k) -\delta_k^{\top}(z^k-x^k) = \triangle \tilde{L}_{\tilde k} -\delta_k^{\top}(z^k-x^k).
\end{eqnarray*}
Let $\xi_{[k-1]}$ be the random samples generated before iteration $k$.
The iterates $x^k$ and $z^{k}$ were determined after $\xi_{[k-1]}$ was realized. This fact combined with Lemma~\ref{smooth-property} implies that
$$
\ex\left[\delta_k^{\top} (z^{k}-x^k)\, \big{|} \, \xi_{[k-1]} \right]=0.
$$
Therefore, $\ex [\triangle {L}_{\tilde k}] \le \ex [\triangle \tilde{L}_{\tilde k}]$, and the theorem is proven.
\end{proof}

\section{Iteration Complexity for Nonconvex Multi-block Optimization}\label{Sec:mul-alg}
In this section we consider the multi-block extension of \eqref{prob:nonconvex}:
\begin{eqnarray}\label{prob:nonconvex-multi-block}\begin{array}{lll} \min & \Phi(x):=f(x_1,\cdots, x_d) + \sum_{i=1}^{d}h_i(x_i) \\ \st & x_i \in S_i \subseteq\RR^{n_i}, \quad i=1,\ldots,d, &  \end{array}\end{eqnarray}
where $f$ is differentiable but possibly nonconvex, and $h_i$ is convex but possibly nonsmooth, $i=1,\ldots, d$, and the feasible region $S_i$ is convex and compact for all $i$ (thus $S=\Pi_{i=1}^{d}S_i$ is compact as well). We denote
$$\mbox{\rm diam}_p(\overline{S}) = \max_{i\in\{1,\ldots,d\}}\max_{x_i,y_i\, \in S_i}\|x_i -y_i \|,\quad \mbox{and}\quad  \mbox{\rm diam}_p(\underline{S}) = \min_{i\in\{1,\ldots,d\}}\max_{x_i,y_i\, \in S_i}\|x_i -y_i \|.$$

A well known technique for solving~\eqref{prob:nonconvex-multi-block} is the so-called block coordinate descent (BCD) method. That is, at each iteration, a single block variable is optimized while all other blocks are fixed. In particular, at iteration $k$ we solve the one-block problem exactly and denote
$$
y_i^k \in \arg\min_{x_i \in S_i} f(x_1^{k-1},\cdots, x_{i-1}^{k-1},x_i,x_{i+1}^{k-1},\cdots,x_{d}^{k-1}) + h_i(x_i),\quad i=1,\ldots, d.
$$
In the classical BCD method with Jacobian updating rule, the blocks are updated cyclicly by setting
$$
x_i^k = y_i^k,\quad i=1,\ldots, d.
$$
Chen et al.~\cite{CHLZ12} proposed another updating rule termed MBI (Maximum Block Improvement), where only the block with maximum improvement is updated at each step. Specifically, we first calculate the maximum improved block
$$
{i_0}\in \arg\max_{i\in \{1,\ldots,d \}} f(x^k) + h_i(x^k_i) - f(x_1^{k-1},\cdots, x_{i-1}^{k-1},y^k_i,x_{i+1}^{k-1},\cdots,x_{d}^{k-1}) - h_i(y^k_i),
$$
and then update the blocks by letting
$$
x_{{i_0}}^k = y_{{i_0}}^k,\quad\mbox{and}\quad x_{{i}}^k = x_{{i}}^{k-1}\quad\mbox{for}\quad i\neq {i_0}.
$$
To differentiate from the Jacobian style updating rule, a cyclic coordinate search is often referred as the BCD method of the Gauss-Seidel type, whose convergence under various settings has been established in~\cite{Tseng01,XuYin2013,RazaviyaynHongLuo13}. Recently, the iteration complexity bounds were successfully established in some convex optimization problems~\cite{BeckTetruashvili13,HongWangRazaviyaynLuo13}. However, computational complexity analysis for {\em nonconvex}\/ multi-block optimization is still very challenging. To the best of our knowledge, Dang and Lan \cite{DangLan2013} was probably the first paper to address this issue through a stochastic approximation method based on the approximation measure introduced in~\eqref{measure-lan}. Here we propose another method, based on the new notion that $x$ is an $\epsilon$-stationary point of~\eqref{prob:nonconvex-multi-block} if
\begin{equation}\label{formu:epsilon-KKT-block}
\nabla_i f(x)^{\top}(y_i-x_i) + h_i(y_i) - h(x_i) \ge - \epsilon\quad \forall y\in S_i,\quad i=1,\ldots,d,
\end{equation}
where $x=(x_1^{\top},\cdots, x_d^{\top})^{\top}$ and $y=(y_1^{\top},\cdots, y_d^{\top})^{\top}$.

In the following, we still assume that inequality~\eqref{ineq:p-power} holds. Like in the single block-variables case as we discussed before, instead of solving the subproblems exactly, we shall use the following partially linearized (and $p$-powered upper bound) functions:
\begin{equation}\label{func:linear-block}
L_i(x_{i};z) := f(z) + \nabla_i f(z)^{\top}(x_i-z_i) + h_i(x_i),\quad i=1,\ldots, d,
\end{equation}
\begin{equation}\label{func:upper-bound-block}
U_i(x_i;z) := f(z) + \nabla_i f(z)^{\top}(x_i-z_i) + \frac{\lambda}{2}\|x_i-z_i \|^p_p + h_i(x_i), \quad i=1,\ldots, d,
\end{equation}
where $z=(z_1^{\top},\cdots, z_d^{\top})^{\top}$. Similar to Lemma~\ref{kkt-criteria}, we have the following criteria to determine the $\epsilon$-stationary point.
\begin{lemma}\label{block-kkt-criteria}Given $\epsilon \ge 0$, for any $z \in S$,\\
(i) if $\triangle_i L_z \le \epsilon$ for $i=1,\cdots, d$, then $z$ is an $\epsilon$-stationary point of~\eqref{prob:nonconvex-multi-block};\\
(ii) if $\triangle_i U_{z} \le \frac{1}{2}\left(\frac{\epsilon}{\mbox{\rm diam}_p(S_i)\lambda^{1/p}}\right)^q$ with $\frac{1}{p} + \frac{1}{q}=1$ and $\epsilon \le \mbox{\rm diam}_p^p(S_i)\lambda$, for $i=1,\cdots, d$,
then $z$ is an $\epsilon$-stationary point of~\eqref{prob:nonconvex-multi-block}.
\end{lemma}

Now we are ready to present our first algorithm for block optimization~\eqref{prob:nonconvex-multi-block}, where either the classical Jacobian updating rule or the MBI updating rule can be applied.
\begin{center}
\begin{tabular}{@{}llr@{}}\toprule
{\bf{Algorithm 5}}\\
\hline
{Let} $x^0 \in S$ {be given} and set $y^0 = x^0$.  \\
 \textbf{for} $k = 1,2,\cdots, N $, \textbf{do}\\
 \qquad \textbf{for} $i = 1,\cdots, d$, \textbf{do}\\
 \qquad \qquad $y^{k}_i=\arg\min_{y_i \in S_i} L_i(y_i;x^{k})$, and let $d^k_i = y^k_i - x^k_i$; \\
 \qquad \qquad $\alpha_{k,i} = \arg\min_{\alpha \in [0,1]}  \alpha\, \nabla_i f(x^{k})^{\top}d^k_i + \alpha^p\,\frac{\lambda}{2}\|d^k_i\|^p_p + (1 - \alpha)h_i(x^k_i) + \alpha\, h_i(y^k_i)$.\\
 \qquad \qquad Set $x^{k+1}_i=(1-\alpha_{k,i})x^{k}_i+ \alpha_{k,i} y^{k}_i$, when \it{Jacobian} updating rule is applied.\\
 \qquad \textbf{end for}\\
\qquad (Or, calculate ${i_0}\in \arg\max_{i\in \{1,\ldots,d \}}\triangle_i L_k$ and update $x^{k+1}_{i_0}=(1-\alpha_{k,{i_0}})x^{k}_{i_0}+ \alpha_{k,{i_0}} y^{k}_{i_0}$,\\
\qquad  $x_{{i}}^k = x_{{i}}^{k-1}\;\mbox{if}\; i\neq {i_0}$ when the MBI updating rule is applied.)\\
\textbf{end for}\\
\hline
\end{tabular}
\end{center}

The computational complexity bound for this algorithm is established as follows:

\begin{theorem}\label{theorm2:CG}
For any $0< \epsilon < \mbox{\rm diam}_p^p(\underline{S}) \lambda$, {\bf Algorithm 5} finds an $\epsilon$-stationary point of~\eqref{prob:nonconvex-multi-block} within $\left\lceil \frac{2(\mbox{\rm diam}_p(\overline{S})^p\lambda)^{q-1}(\Phi(x^{1}) - \Phi^*)}{\epsilon^{q}} \right\rceil$ steps.
\end{theorem}
\begin{proof} Since inequality~\eqref{ineq:p-power} holds, one has
\begin{eqnarray} \label{formu-in-proof-2}
&&\sum_{i=1}^{d}\left(-\nabla_i f(x^{k})^{\top}(x^{k+1}_i-x^{k}_i) - \frac{\lambda}{2}\|x^{k+1}_i - x^{k}_i\|^p_p+ h_i(x^{k}_i)-h_i(x^{k+1}_i)\right) \nonumber \\
&\le& f(x^{k}) - f(x^{k+1}) + \sum_{i=1}^{d}\left(h_i(x^{k}_i)-h_i(x^{k+1}_i)\right) \nonumber \\
&=& \Phi(x^{k}) - \Phi(x^{k+1}) .
\end{eqnarray}
When the Jacobian updating rule is applied, by the definition of $\alpha_{k,i}$ we have
$$
\alpha_{k,i} \left( \nabla_i f(x^{k})^{\top}(y^{k}_i-x^{k}_i)+h_i(y^{k}_i) - h_i(x^{k}_i) \right) + \frac{\lambda \alpha_{k,i}^p}{2}\|y^{k}_i - x^{k}_i\|^p_p \le 0.
$$
Recall that ${i_0}\in \arg\max_{i\in \{1,\ldots,d \}}\triangle_i L_k$. Therefore,
\begin{equation} \label{formu-in-proof-21}
\begin{array}{rcl}
&&-\alpha_{k,i_0} \left( \nabla_{i_0} f(x^{k})^{\top}(y^{k}_{i_0}-x^{k}_{i_0})+h_{i_0}(y^{k}_{i_0}) - h_{i_0}(x^{k}_{i_0}) \right) - \frac{\lambda \alpha_{k,{i_0}}^p}{2}\|y^{k}_{i_0} - x^{k}_{i_0}\|^p_p  \\
&\le&\sum_{i=1}^{d}\left(  -\alpha_{k,i} \left( \nabla_i f(x^{k})^{\top}(y^{k}_i-x^{k}_i)+h_i(y^{k}_i) - h_i(x^{k}_i) \right) - \frac{\lambda \alpha_{k,i}^p}{2}\|y^{k}_i - x^{k}_i\|^p_p \right)  \\
&=& \sum_{i=1}^{d}\left( -\nabla_i f(x^{k})^{\top}(\alpha_{k,i} (y^{k}_i-x^{k}_i)) + h_i(x^{k}_i)-(1 - \alpha_{k,i})h_i(x^k_i) - \alpha_{k,i} h_i(y^k_i) - \frac{\lambda }{2}\|\alpha_{k,i}(y^{k}_i - x^{k}_i)\|^p_p \right) \\
&\le & \sum_{i=1}^{d}\left(-\nabla_i f(x^{k})^{\top}(x^{k+1}_i-x^{k}_i) - \frac{\lambda}{2}\|x^{k+1}_i - x^{k}_i\|^p_p+ h_i(x^{k}_i)-h_i(x^{k+1}_i)\right), \\
\end{array}
\end{equation}
where the last inequality follows from the convexity of function $h_i(\cdot)$.

On the other hand, in the case of the MBI updating rule, we have
\begin{eqnarray*}
&& \sum_{i=1}^{d}\left(-\nabla_i f(x^{k})^{\top}(x^{k+1}_i-x^{k}_i) - \frac{\lambda}{2}\|x^{k+1}_i - x^{k}_i\|^p_p+ h_i(x^{k}_i)-h_i(x^{k+1}_i)\right)\\
&=& -\nabla_{i_0} f(x^{k})^{\top}(x^{k+1}_{i_0}-x^{k}_{i_0}) - \frac{\lambda}{2}\|x^{k+1}_{i_0} - x^{k}_{i_0}\|^p_p+ h_{i_0}(x^{k}_{i_0})-h_{i_0}(x^{k+1}_{i_0}).
\end{eqnarray*}
Thus, according to the convexity of $h_{i_0}(\cdot)$, inequality~\eqref{formu-in-proof-21} still holds.

Therefore, in either cases, by~\eqref{formu-in-proof-2},~\eqref{formu-in-proof-21} and the optimality of $\alpha_{k,{i_0}}$ we have
\begin{eqnarray*}
\left(\frac{\epsilon}{\mbox{\rm diam}_p^p(\overline{S})\lambda}\right)^{\frac{1}{p-1}}\max_{i}\triangle_i L_k &\le& \Phi(x^{k}) - \Phi(x^{k+1}) + \frac{\epsilon^{\frac{p}{p-1}} \lambda \mbox{\rm diam}_p^p(S_{i_0}) }{2(\mbox{\rm diam}_p^p(\overline{S})\lambda)^{\frac{p}{p-1}}}\\
&\le&\Phi(x^{k}) - \Phi(x^{k+1}) + \frac{\epsilon^{\frac{p}{p-1}}}{2(\mbox{\rm diam}_p^p(\overline{S})\lambda)^{\frac{1}{p-1}}}.
\end{eqnarray*}
Summing up the above inequality for $k=1,\ldots,N$ yields
\begin{eqnarray*}
\left(\frac{\epsilon}{\mbox{\rm diam}_p^p(\overline{S})\lambda}\right)^{\frac{1}{p-1}} N \min_{k\in\{1,\ldots,N\}}\max_{i}\triangle_i L_k  &\le& \left(\frac{\epsilon}{\mbox{\rm diam}_p^p(\overline{S})\lambda}\right)^{\frac{1}{p-1}} \sum_{k=1}^{N}\max_{i}\triangle_i L_k  \\
&\le& \Phi(x^{1}) - \Phi(x^{N+1})+ N\frac{\epsilon^{\frac{p}{p-1}}}{2(\mbox{\rm diam}_p^p(\overline{S})\lambda)^{\frac{1}{p-1}}}\\
 &\le& \Phi(x^{1}) - \Phi^*+ N\frac{\epsilon^{\frac{p}{p-1}}}{2(\mbox{\rm diam}_p^p(\overline{S})\lambda)^{\frac{1}{p-1}}}.
\end{eqnarray*}
Thus, if
$$N = \left\lceil \frac{2(\mbox{\rm diam}_p(\overline{S})^p\lambda)^{q-1}(\Phi(x^{1}) - \Phi^*)}{\epsilon^{q}} \right\rceil,
$$
then dividing both sides by $N \left(\frac{\epsilon}{\mbox{\rm diam}_p^p(\overline{S})\lambda}\right)^{\frac{1}{p-1}}$, we conclude that there must exist some $\tilde{k} \le N$ such that
$$
\triangle_i L_{\tilde{k}} \le \epsilon,\quad \mbox{for all}\quad i=1,\ldots, d,
$$
which combined with Lemma~\ref{block-kkt-criteria} implies that $x^{\tilde{k}}$ is an $\epsilon$-stationary point of~\eqref{prob:nonconvex-multi-block}.
\end{proof}

Our second algorithm to solve problem~\eqref{prob:nonconvex-multi-block} uses an upper bound for the objective function. Again, we can either apply the classical Jacobian updating rule or the MBI updating rule.
\begin{center}
\begin{tabular}{@{}llr@{}}\toprule
{\bf{Algorithm 6}}\\
\hline
{Let} $x^1 \in S$ {be given}  \\
 \textbf{for} $k = 1,2,\cdots, N $, \textbf{do}\\
 \qquad $y^{k+1}_i=\arg\min_{y_i \in S_i} U_i(y_i;x^{k})$ for $i=1,\ldots, d$.\\
 \qquad Set $x_i^{k+1} = y_i^{k+1},\, i=1,\ldots, d$, when \it{Jacobian} updating rule is applied.\\
 \qquad (Or, calculate ${i_0}\in \arg\max_{i\in \{1,\ldots,d \}}\triangle_i U_k$ and update $x_{{i_0}}^{k+1} = y_{{i_0}}^{k+1}, x_{{i}}^{k+1} = x_{{i}}^{k}\;\mbox{if}\; i\neq {i_0}$\\
 \qquad  when the {\it MBI} updating rule is applied.)\\
\textbf{end for}\\
\hline
\end{tabular}
\end{center}

Similarly, the computational complexity result for that algorithm can be established as follows.
\begin{theorem}\label{theorm:PPA-block} For any $0< \epsilon < \mbox{\rm diam}_p^p(\underline{S}) \lambda$, {\bf Algorithm 6} finds an $\epsilon$-stationary point of~\eqref{prob:nonconvex-multi-block} within $\left\lceil \frac{2(\mbox{\rm diam}_p(\overline{S})^p\lambda)^{q-1}(\Phi(x^{1}) - \Phi^*)}{\epsilon^{q}} \right\rceil$ steps.
\end{theorem}

Similar arguments as in the proof of Theorem~\ref{theorm2:CG} can be used to prove the above theorem; we leave the details to the interested reader.

\section{Numerical Experiments} \label{numerical}
In this section, we provide numerical performance of our algorithms for solving two nonconvex problems: the problem of finding the leading sparse principle component of tensor and the penalized zero-variance linear discriminant analysis.
\subsection{Computing the Leading Sparse Principle Component of Tensor}
The problem of finding the principle component (PC) that explains the most variance of a tensor $\mathcal{A}$ (with degree $d$) can be formulated as:
\begin{equation*}
\begin{array}{ll}
\min & \| \mathcal{A} - \lambda x_1\otimes x_2 \otimes \cdots \otimes x_d \|^2_2 \\
\mbox{s.t.} & \lambda\in \RR,\, \| x_i\|_2^2 =1, i=1,2,\ldots,d,
\end{array}
\end{equation*}
where `$\otimes$' is the tensor outer-product operation.
This problem bears different names including the tensor best rank-one approximation~\cite{KofidisRegalia02} and the Z-eigenvalue problem~\cite{Qi05,Lim05}, and various solution methods have been proposed:~\cite{CHLZ12,JiangMaZhang14,KofidisRegalia02,QiWangWang09}.

Like in the matrix case, sparsity is desirable in tensor decomposition under various environments~\cite{Allen12}. Consider the following sparse tensor PCA problem:
\begin{equation}\label{prob:sparse-PCA}
\begin{array}{ll}
\min & \| \mathcal{A} - \lambda x_1\otimes x_2 \otimes \cdots \otimes x_d \|^2_2 + \rho\sum\limits_{i=1}^{d}\|x_i\|_0 \\
\mbox{s.t.} & \lambda\in \RR,\, \| x_i\|_2^2 =1, i=1,2,\ldots,d,
\end{array}
\end{equation}
which is equivalent to
\begin{equation*}
\begin{array}{ll}
\min & -\mathcal{A} (x_1, x_2 , \cdots , x_d) + \rho\sum\limits_{i=1}^{d}\|x_i\|_0  \\
\mbox{s.t.} & \| x_i\|_2 \le 1, i=1,2,\ldots,d.
\end{array}
\end{equation*}
To apply the algorithms discussed in the previous section, we replace $\|\cdot\|_0$ by $\|\cdot\|_1$, and arrive at the following formulation
\begin{equation}\label{prob:tensorPCA-rx1}
\begin{array}{ll}
\min & -\mathcal{A} (x_1, x_2 , \cdots , x_d) + \rho\sum\limits_{i=1}^{d}\|x_i\|_1 \\
\mbox{s.t.} & \| x_i\|_2 \le 1, i=1,2,\ldots,d.
\end{array}
\end{equation}

Denote the matrix
$$
\mathcal{A}(x^{-ij}) := \mathcal{A}(x_1,\cdots,x_{j-1},\cdot,x_{j+1},\cdots,x_{i-1},\cdot,x_{i+1},\cdots,x_d),
$$
and let $\tau = \max_{\|x\|_2 \le 1}\| \mathcal{A}(x^{-ij}) \|_2 $.
Then for any $x=(x_1^{\top},\cdots, x_d^{\top})^{\top}, y=(y_1^{\top},\cdots, y_d^{\top})^{\top}$, and index $i$, we have
\begin{eqnarray*}
&& \|\mathcal{A}(x_1,\cdots,x_{i-1},\cdot,x_{i+1},\cdots,x_d) - \mathcal{A}(y_1,\cdots,y_{i-1},\cdot,y_{i+1},\cdots,y_d)\|_2 \\
&=& \bigg{\|}\sum_{j\neq i} \mathcal{A}(y_1,\cdots,y_{j-1},x_{j},\cdots,x_{i-1},\cdot,x_{i+1},\cdots,x_d) - \mathcal{A}(y_1,\cdots,y_{j},x_{j+1},\cdots,x_{i-1},\cdot,x_{i+1},\cdots,x_d)\bigg{\|}_2\\
&\le& \sum_{j\neq i}  \|\mathcal{A}(y_1,\cdots,y_{j-1},x_{j},\cdots,x_{i-1},\cdot,x_{i+1},\cdots,x_d) - \mathcal{A}(y_1,\cdots,y_{j},x_{j+1},\cdots,x_{i-1},\cdot,x_{i+1},\cdots,x_d)\|_2\\
&\le & \tau \sum_{j\neq i} \|x_j - y_j \|_2.
\end{eqnarray*}
Consequently,
$$
\| \nabla \mathcal{A}(x) - \nabla \mathcal{A}(y) \|_2 \le \tau (d-1) \sum_{i=1}^{d}\|x_i - y_i \|_2 \le \tau d(d-1) \|x-y\|_2,
$$
which means that Assumption~\ref{assump-p} holds for $p=q=2$. Therefore our {\bf{Algorithm 5}} and {\bf{Algorithm 6}} can be applied to solve problem~\eqref{prob:tensorPCA-rx1}. When {\bf{Algorithm 5}} is applied, the subproblem is in the form of
\begin{equation}\label{subprob:sparce-tensor-PCA}\min_{\|y\|_2^2 \le 1} \{ -y^{\top}b + \rho\|y\|_1 \}.
\end{equation}
Denote $z(j)=\sign(b(j))\max\{ |b(j)| - \rho, 0 \}\;\forall j$. Problem~\eqref{subprob:sparce-tensor-PCA} has a closed form solution
\[
y^* = \left\{\begin{array}{ll}
z/\|z\|_2, & \mbox{if}\; \|z\|_2 \neq 0 \\
0, & \mbox{otherwise.}
\end{array}\right.
\]
In {\bf{Algorithm 6}} the subproblem under consideration is given by $\min_{\|y\|_2^2 \le 1} \{ -y^{\top}b + \rho\|y\|_1 + \frac{\lambda}{2}\|y\|^2_2 \}$, which has a closed form solution
$y^* = \frac{z}{\lambda + \max\{ 0, \|z\|_2-\lambda \}}$.

One undesirable property of the relaxed formulation~\eqref{prob:tensorPCA-rx1} is that we may possibly get a zero solution; i.e.\ $x_i=0$ for some $i$, which leads to $\mathcal{A} (x_1, x_2 , \cdots , x_d)=0$. To prevent this from happening, we also apply the BCD method with the Jacobian updating rule to the following equality constraint problem:
\begin{equation}\label{prob:tensorPCA-rx2}
\begin{array}{ll}
\min & -\mathcal{A} (x_1, x_2 , \cdots , x_d) + \rho\sum\limits_{i=1}^{d}\|x_i\|_1 \\
\mbox{s.t.} & \| x_i\|_2 = 1, i=1,2,\ldots,d,
\end{array}
\end{equation}
and compare the results with those returned by our proposed algorithms in Table~\ref{tab:tensor-pca}.

In the tests, we let $\lambda = 20$, $\rho=0.85$,
and set the maximum iteration number to be $2000$; we only apply the Jacobian updating rule in the implementation of {\bf Algorithm 5} and {\bf Algorithm 6}.
For each fixed dimension, we randomly generate $10$ instances which are the fourth order tensors and the corresponding problems are solved by the three methods, starting from the same initial point.
In Table~\ref{tab:tensor-pca}, `Val.' refers to the value $\mathcal{A} (x_1, x_2 , \cdots , x_d)$. From this table,
we see that {\bf Algorithm~5} is the most stable method for the sparse tensor PCA problem, as it is able to find a nonzero local minimum within a few hundred steps in most cases, with reasonably sparse solutions. For the same collection of instances, {\bf Algorithm 6} falls into the zero solution in many cases, while the BCD method for~\eqref{prob:tensorPCA-rx2} on the other hand, ends up in a local minimum point instead of local maximum for quite a few instances.
\begin{table}[ht]{\footnotesize
\centering
\begin{tabular}{|c|c|c|c|c|c|c|c|c|c|}
\hline
Inst.  \#&\multicolumn{3}{|c|}{ BCD}&\multicolumn{3}{|c|}{Algorithm $6$}&\multicolumn{3}{|c|}{Algorithm $5$}\\
\hline
&Val.& $\sum\limits_{i=1}^{d}\|x_i\|_0$& Iter.& Val.& $\sum\limits_{i=1}^{d}\|x_i\|_0$& Iter.&Val.& $\sum\limits_{i=1}^{d}\|x_i\|_0$& Iter. \\\hline
\multicolumn{10}{|c|}{Dimension $n=8$}\\\hline
 1 & 4.99 & 20 & 2000  & 7.76 & 21 &  99 & 6.76 & 18 &  92\\
 2 & 6.42 & 14 &  96  & 0.00 &  0  &  20 & 0.00 &  8 &  13\\
 3 & -6.41 & 16 &  46  & 6.56 & 16 & 213 & 6.56 & 16 & 139\\
 4 & -7.16 & 17 &  83  & 0.00 &  0  &  24 & 0.00 & 16 &  15\\
 5 & -6.36 & 18 &  99  & 6.78 & 18 & 185 & 6.19 & 15 &  91\\
 6 & -8.36 & 20 & 132  & 8.98 & 22 & 156 & 8.98 & 22 & 148\\
 7 & 6.10 & 18 & 2000  & 7.51 & 18 & 154 & 7.51 & 18 &  98\\
 8 & -7.14 & 20 &  95  & 7.14 & 20 & 357 & 7.14 & 20 & 295\\
 9 & 8.73 & 21 & 127  & 8.50 & 21 & 227 & 8.50 & 21 & 155\\
10 & 6.25 & 16 & 327  & 7.41 & 19 & 207 & 7.41 & 19 & 119 \\
 \hline
\multicolumn{10}{|c|}{Dimension $n=12$}\\\hline
 1 & -7.80 & 20 & 2000  & 8.34 & 27 & 841 & 8.23 & 23 & 166 \\
 2 & -9.52 & 31 & 293  & 8.16 & 23 & 222 & 8.73 & 26 & 315\\
 3 & 9.17 & 28 & 282  & 10.19 & 33 & 528 & 10.19 & 33 & 362\\
 4 & 8.50 & 22 & 257  & 7.63 & 21 & 230 & 8.74 & 22 & 296\\
 5 & -9.58 & 22 & 2000  & 8.61 & 23 & 314 & 8.61 & 24 & 221\\
 6 & 9.95 & 28 & 267  & 8.48 & 22 & 211 & 8.48 & 22 & 151\\
 7 & 8.88 & 23 & 142  & 0.00 &  0  &  18 & 0.00 & 24 &  11\\
 8 & -8.42 & 27 & 263  & 8.55 & 27 & 250 & 8.55 & 27 & 154\\
 9 & -8.64 & 26 & 2000  & 8.81 & 24 & 166 & 8.49 & 31 &  53\\
10 & 9.54 & 30 & 208  & 8.89 & 26 & 131 & 8.89 & 26 &  97 \\
\hline
\multicolumn{10}{|c|}{Dimension $n=20$}\\\hline
 1 & 6.80 & 39 & 2000  & 0.00 &  0  &  31 & 11.52 & 41 & 290 \\
 2 & -12.95 & 49 & 278  & 0.00 & 0   &  17 & 11.77 & 44 & 112\\
 3 & -11.44 & 38 & 277  & 11.08 & 39 & 156 & 12.71 & 42 & 195\\
 4 & -11.22 & 40 & 766  & 0.00 &  0  &  20 & 11.50 & 39 & 141\\
 5 & 11.51 & 38 & 1267  & 0.00 &  0  &  18 & 0.00 & 0   &  11\\
 6 & 12.42 & 44 & 808  & 11.44 & 36 & 225 & 11.47 & 36 & 144\\
 7 & -12.22 & 47 & 2000  & 11.28 & 34 & 269 & 13.20 & 49 & 241\\
 8 & 11.35 & 39 & 2000  & 11.03 & 35 & 191 & 10.80 & 40 & 211 \\

  9 & 11.74 & 44 & 2000  & 11.88 & 37 & 194 & 12.34 & 47 & 199\\
10 & -11.49 & 46 & 493  & 11.60 & 42 & 172 & 11.79 & 43 & 454\\
\hline
\multicolumn{10}{|c|}{Dimension $n=30$}\\\hline
 1 & 14.13 & 52 & 1673  & 0.00 & 0   &  12 & 15.22 & 58 & 311\\
 2 & 0.82 & 41 & 2000  & 0.00 &  0 &  25 & 14.36 & 53 & 214\\
 3 & -15.23 & 59 & 1589  & 0.00 &  0  &  15 & 14.28 & 55 & 238\\
 4 & 0.25 & 59 & 2000  & 0.00 &  0  &  15 & 0.00 &  0  &   9\\
 5 & 0.74 & 43 & 2000  & 14.36 & 56 & 175 & 14.36 & 56 & 101\\
 6 & 0.53 & 35 & 2000  & 13.33 & 47 & 454 & 15.24 & 58 & 267\\
 7 & 0.48 & 43 & 2000  & 13.19 & 51 & 377 & 13.39 & 51 & 162\\
 8 & 0.38 & 40 & 2000  & 0.00 &  0  &  14 & 0.00 &  0  &   9\\
 9 & -13.32 & 48 & 1523  & 13.93 & 56 & 692 & 12.36 & 47 & 374\\
10 & 11.84 & 40 & 1034  & 13.75 & 51 & 455 & 15.18 & 61 & 361 \\
\hline
\end{tabular}
\caption{Numerical results for sparse tensor PCA problem.}
\label{tab:tensor-pca}}
\end{table}

\subsection{Penalized Zero-Variance Linear Discriminant Analysis}
The penalized version of the zero-variance discriminant analysis is presented by Ames and Hong~\cite{AmesHong14}, aiming to perform linear discriminant analysis and feature selection on high-dimensional data simultaneously. To be more specific, the problem under consideration can be described as follows:
\begin{equation}\label{prob:SZVD}
\min_{x \in \RR^n,\;x^{\top}x \le 1}-\frac{1}{2}x^{\top}N^{\top}BNx + \gamma\sum_{i=1}^{m}\sigma_i | (DNX)_i |,
\end{equation}
where $B \in \RR^{m \times m}$ is a positive semidefinite matrix, $D \in \RR^{m \times m}$ and $N \in \RR^{m \times n}$ are orthogonal matrices. Thus, the objective is the summation of a concave function and a convex function. Due to Corollary~\ref{coro:concave}, {\bf Algorithm 1} with $\alpha_k = 1$ for all $k$ can achieve $O(1/\epsilon)$ iteration complexity. In particular, at the $k$-th iteration, based on point $x^k$ we can find the next point $x^{k+1}$ by optimizing a homogeneous convex problem:
$$
x^{k+1} := \arg_{x \in \RR^n} -(x^k)^{\top}N^{\top}BNx + \gamma\sum_{i=1}^{m}\sigma_i | (DNX)_i |,
$$
which can be solved by CVX~\cite{cvx} efficiently. While Ames and Hong~\cite{AmesHong14} proposes {\it alternating direction method of multipliers} (ADMM) to solve~\eqref{prob:SZVD}. The advantage of that approach is that the subproblems have closed form solutions.

For fixed dimension, we compare our approach with ADMM in~\cite{AmesHong14} for $10$ randomly generated instances, and the results are provided in Table~\ref{tab:SZVD}. It appears that these two methods produce sequences that converge to the same point, but {\bf Algorithm 1} costs much less iterations.
\begin{table}[ht]{
\centering
\begin{tabular}{|c|c|c|c|c|}
\hline
Inst.  \#&\multicolumn{2}{|c|}{Algorithm 1}&\multicolumn{2}{|c|}{ADMM}\\
\hline
&Obj.Val.&Iter.& Obj.Val.&Iter. \\\hline
\multicolumn{5}{|c|}{Dimension $n=50$, $m=100$}\\\hline
 1 & -279.372 &   93 & -279.370 & 437 \\
 2 & -273.868 &   61 & -273.868 & 739 \\
 3 & -263.741 &   72 & -263.743 & 426 \\
 4 & -257.462 &   82 & -257.465 & 566 \\
 5 & -263.832 &  114 & -263.848 & 652 \\
 6 & -291.784 &   42 & -291.784 & 220 \\
 7 & -277.502 &   46 & -277.503 & 250 \\
 8 & -279.291 &   65 & -279.297 & 343 \\
 9 & -265.733 &  145 & -265.741 & 1029 \\
 10 & -269.525 &   62 & -269.523 & 273 \\\hline
\multicolumn{5}{|c|}{Dimension $n=100$, $m=200$}\\\hline
 1 & -574.195 &  133 & -574.199 & 728 \\
 2 & -553.266 &  78 & -553.270 &  587 \\
 3 & -564.459 &  161 & -564.456 &  648\\
 4 & -586.036 &   46 & -586.039 &  240\\
 5 & -554.430 &  135 & -554.447 & 880 \\
 6 & -563.877 &   59 & -563.883 & 308 \\
 7 & -559.835 &   72 & -559.844 & 434 \\
 8 & -557.033 &  157 & -557.042 & 823 \\
 9 & -558.926 &  155 & -558.929 & 674 \\
 10 & -573.841 &  198 & -573.861 & 1019 \\\hline
 \multicolumn{5}{|c|}{Dimension $n=200$, $m=400$}\\\hline
 1 & -1172.918 &   85 & -1172.916 & 638  \\
 2 & -1154.821 &  168 & -1154.834 & 993 \\
 3 & -1139.545 &  129 & -1139.554 & 662 \\
 4 & -1121.581 &  245 & -1121.587 & 922 \\
 5 & -1176.465 &  116 & -1176.465 & 640 \\
 6 & -1149.466 &  106 & -1149.472 & 521\\
 7 & -1151.241 &   90 & -1151.245 & 417 \\
 8 & -1156.047 &  339 & -1156.068 & 1533 \\
 9 & -1135.454 &   59 & -1135.518 & 474 \\
 10 & -1140.661 &  512 & -1140.704 & 2840 \\
\hline
\end{tabular}
\caption{Numerical results for Penalized Zero-Variance Linear Discriminant Analysis.}
\label{tab:SZVD}}
\end{table}

\end{document}